\documentclass[11pt,oneside]{amsart}

\usepackage[T1]{fontenc} 
\usepackage[utf8]{inputenc}  
\usepackage[english]{babel}

\allowdisplaybreaks

\usepackage{vmargin}
\setmarginsrb{3cm}{2cm}{3cm}{2cm}{0cm}{2cm}{0cm}{1cm}

\usepackage{tikz}
\usetikzlibrary{automata,arrows.meta}
\definecolor{MyGreen}{rgb}{0.13,0.55,0.13}

\theoremstyle{plain}
\newtheorem{theorem}{Theorem}
\newtheorem{lemma}[theorem]{Lemma}
\newtheorem{corollary}[theorem]{Corollary}
\newtheorem{proposition}[theorem]{Proposition}
\theoremstyle{definition}

\newtheorem{example}[theorem]{Example}
 
\numberwithin{equation}{section}

\newcommand{\R}{\mathbb R}
\newcommand{\N}{\mathbb N}
\newcommand{\Z}{\mathbb Z}
\newcommand{\Q}{\mathbb Q}
\newcommand{\C}{\mathbb C}
\newcommand{\pr}{\delta}

\newcommand{\A}{\mathcal{A}}
\newcommand{\B}{\boldsymbol{\beta}}
\newcommand{\Dig}{\mathcal{D}}
\newcommand{\val}{\mathrm{val}}
\newcommand{\lex}{\mathrm{lex}}
\newcommand{\Fac}{\mathrm{Fac}}
\newcommand{\Pref}{\mathrm{Pref}}
\newcommand{\ceil}[1]{\left\lceil#1\right\rceil}
\newcommand{\floor}[1]{\left\lfloor#1\right\rfloor}
\newcommand{\Int}{[\![0,p-1]\!]}
\newcommand{\DB}{d_{\boldsymbol{\beta}}}
\newcommand{\DBi}[1]{d_{{\B}^{(#1)}}}
\newcommand{\qDB}{d_{\boldsymbol{\beta}}^{*}}
\newcommand{\qDBi}[1]{d_{\boldsymbol{\beta}^{(#1)}}^{*}}
\newcommand{\couple}[2]{\left[\begin{smallmatrix} #1 \\ #2 \end{smallmatrix}\right]}

\title{Spectrum, algebraicity and normalization in alternate bases}
\author{\'Emilie Charlier$^{1,*}$, Célia Cisternino$^{1}$, Zuzana Mas\'akov\'a$^2$ and Edita Pelantov\'a$^2$}
\address{$^1$Department of Mathematics\\
University of Li\`ege\\
All\'ee de la D\'ecouverte 12,
4000 Li\`ege, Belgium\\
$^2$Department of Mathematics
\\
Czech Technical University in Prague\\
Trojanova 13, 120 00 Praha 2, Czech Republic
}
\thanks{\emph{E-mail address:} \texttt{echarlier@uliege.be, ccisternino@uliege.be zuzana.masakova@fjfi.cvut.cz\\ and edita.pelantova@fjfi.cvut.cz}\\
$^*$Corresponding author.}

\begin{document}
\tikzset{elliptic state/.style={draw,ellipse,minimum width=6cm,minimum height=1.5cm}}

\begin{abstract}
The first aim of this article is to give information about the algebraic properties of alternate bases $\B=(\beta_0,\dots,\beta_{p-1})$ determining sofic systems. We show that a necessary condition is that the product $\pr=\prod_{i=0}^{p-1}\beta_i$ is an algebraic integer and all of the bases $\beta_0,\ldots,\beta_{p-1}$ belong to the algebraic field $\Q(\pr)$. On the other hand, we also give a sufficient condition: if $\pr$ is a Pisot number and $\beta_0,\ldots,\beta_{p-1}\in\Q(\pr)$, then the system associated with the alternate base $\B=(\beta_0,\dots,\beta_{p-1})$ is sofic. The second aim of this paper is to provide an analogy of Frougny's result concerning normalization of real bases representations. We show that given an alternate base $\B=(\beta_0,\dots,\beta_{p-1})$ such that $\pr$ is a Pisot number and $\beta_0,\ldots,\beta_{p-1}\in\Q(\pr)$, the normalization function is computable by a finite Büchi automaton, and furthermore, we effectively construct such an automaton. An important tool in our study is the spectrum of numeration systems associated with alternate bases. The spectrum of a real number $\delta>1$ and an alphabet $A\subset {\mathbb Z}$ was introduced by Erd\H os et al. For our purposes, we use a generalized  concept with $\delta\in{\mathbb C}$ and $A\subset{\mathbb C}$ and study its topological properties.
\end{abstract}

\maketitle

\bigskip
\hrule
\bigskip



\bigskip
\hrule
\bigskip

\section{Introduction}

Alternate bases are particular cases of Cantor real bases, which were introduced by the first two authors in~\cite{CharlierCisternino2021} and then studied in~\cite{CharlierCisterninoDajani2021,Cisternino2021}. A Cantor real base is a sequence $\B=(\beta_n)_{n\in \N}$ of real numbers greater than $1$ such that $\prod_{n=0}^{+\infty}\beta_n=+\infty$. A $\B$-representation of a real number $x\in[0,1]$ is an infinite sequence $a_0a_1a_2\cdots$ of non-negative integers such that
\[
	x=\sum_{n=0}^{+\infty} \frac{a_n}{\prod_{k=0}^{n}\beta_k}.
\]
When choosing a constant sequence of bases $\B=(\beta,\beta,\ldots)$, we obtain the well-known R\'enyi numeration system associated with one real base $\beta$~\cite{Renyi1957}. Moreover, when choosing a sequence of integer bases $\B=(b_n)_{n\in \N}$, we obtain the so-called Cantor expansions of real numbers~\cite{Cantor1869}. Thus, Cantor real bases generalize both expansions in a real base and Cantor expansions. 

Two particular $\B$-representations of $x\in [0,1]$ can be obtained thanks to the greedy and lazy algorithms. In~\cite{CharlierCisternino2021}, the classical combinatorial results of the greedy $\beta$-expansions were generalized to the framework of Cantor re~al bases. In~\cite{Cisternino2021}, the lazy combinatorial properties of Cantor real bases were investigated. Moreover, in~\cite{CharlierCisterninoDajani2021}, the classical dynamical results of greedy and lazy $\beta$-expansions were generalized while focusing on periodic Cantor real bases $\B=(\beta_0,\ldots,\beta_{p-1},\beta_0,\ldots,\beta_{p-1},\ldots)$. Such Cantor real bases are called alternate bases and are simply denoted by $\B=(\beta_0,\ldots,\beta_{p-1})$. Since one base numeration systems as defined by R\'enyi have been studied extensively from many aspects, a lot of other questions can be investigated in the framework of Cantor real bases and in particular, of alternate bases. In this paper, we study some algebraic properties of alternate base expansions. 

Representations of real numbers involving more than one base simultaneously and independently aroused the interest of other mathematicians~\cite{CaalimaDemegillo2020,
Li2021,
Neunhauserer2021,
ZouKomornikLu2021}.  
But so far, most of the research was concentrated on the combinatorial properties of these representations as in~\cite{CharlierCisternino2021} and not on their dynamical or algebraic properties which are studied respectively in~\cite{CharlierCisterninoDajani2021} and in this paper.

The dynamical point of view of real base expansions was the focus of Bertrand-Mathis~\cite{Bertrand-Mathis1986} who showed that the numeration system with base $\beta>1$ defines a $\beta$-shift which is sofic, i.e., its factors form a language that is accepted by a finite automaton, if and only if the $\beta$-expansion of $1$ is eventually periodic. Real bases $\beta$ determining sofic $\beta$-shifts are called Parry numbers. An algebraic description of Parry numbers is not obvious. It is known that every Parry number is a Perron number, i.e., an algebraic integer $\beta>1$ whose conjugates are in modulus smaller than $\beta$. On the other hand, the set of Parry numbers includes Pisot numbers, i.e., algebraic integers greater than $1$ with conjugates inside the unit circle~\cite{Bertrand1977}. More detailed information on algebraic conjugates of a Parry number $\beta$ was given by Solomyak~\cite{Solomyak1994}.

An analogue of the result of Bertrand-Mathis for alternate bases $\B=(\beta_0,\dots,\beta_{p-1})$ was proved in~\cite{CharlierCisternino2021}. Namely, it is proved that the alternate base system defines a sofic $\B$-shift if and only if each of the $p$ greedy $\B^{(i)}$-expansions of $1$ is eventually periodic where $\B^{(i)}=(\beta_i,\ldots,\beta_{p-1},\beta_0,\ldots,\beta_{i-1})$. The first aim of this article is to give information about the algebraic properties of alternate bases $\B=(\beta_0,\dots,\beta_{p-1})$ determining sofic systems. In particular, in Theorem~\ref{Thm : AlleventuallyPeriodicAlgebraicField} we show a necessary condition, namely that the product $\pr=\prod_{i=0}^{p-1}\beta_i$ is an algebraic integer and all of the bases $\beta_0,\ldots,\beta_{p-1}$ belong to the algebraic field $\Q(\pr)$. On the other hand, in Theorem~\ref{Thm : PisotExtendedFieldThenUltPer}, we also give a sufficient condition: if $\pr$ is a Pisot number and $\beta_0,\ldots,\beta_{p-1}\in\Q(\pr)$, then the $\B$-shift is sofic.

The importance of the class of Pisot bases in connection to automata was pointed out also by Frougny~\cite{Frougny1992} who showed that normalization in a real base $\beta>1$ which maps any $\beta$-representation of a real number in $[0,1)$ to its greedy $\beta$-expansion is computable by a finite Büchi automaton if $\beta$ is a Pisot number. 
The second aim of this paper is to provide an analogy of Frougny's result concerning normalization. In Theorem~\ref{Thm : Normalization}, we show that given an alternate base $\B=(\beta_0,\dots,\beta_{p-1})$ such that $\pr=\prod_{i=0}^{p-1}\beta_i$ is a Pisot number and $\beta_0,\ldots,\beta_{p-1}\in\Q(\pr)$, the normalization function is computable by a finite Büchi automaton.

An important tool in our proofs is the spectrum of numeration systems associated with alternate bases. Its definition shows that one needs to consider the spectrum of $\pr=\prod_{i=0}^{p-1}\beta_i$ with a more general alphabet of non-integer digits. Hence, we first study the spectrum in the general framework of a complex base $\pr$ such that $|\pr|>1$ with an alphabet $A\subset\C$, which is defined as
\[
	X^A(\pr)=\{\sum_{i=0}^{n}a_i\pr^i : n\in\N, a_i\in A\}.
\]
The notion of spectrum was originally introduced by Erd\H os, Jo\'o and Komornik for a base $\pr\in(1,2)$ and an alphabet of the form $A=\{0,1,\dots,m\}$~\cite{ErdosJooKomornik1990}. Topological properties of the spectrum determine many of the arithmetical aspects of the numeration system; see~\cite{FrougnyPelantova2018}. One of the main problems in the study of spectra is to describe bases which give spectra without accumulation points in dependence on the alphabet. For the case of real bases and symmetric integer alphabets, a complete characterization was given by Akiyama and Komornik~\cite{AkiyamaKomornik2013} and Feng~\cite{Feng2016}. 

In this paper, as an analogy to the results of~\cite{FrougnyPelantova2018}, we prove in Theorem~\ref{Thm : SetZeroSpectrumAccZeroAutomatonComplex} that the set of representations of zero in a complex base $\pr$ such that $|\pr|>1$ and an alphabet $A$ of complex number is accepted by a finite Büchi automaton if and only if the spectrum $X^A(\pr)$ has no accumulation point. Next, we deduce an analogue of Theorem~\ref{Thm : SetZeroSpectrumAccZeroAutomatonComplex} in the alternate base case, namely Theorem~\ref{Thm : SetZeroSpectrumAccZeroAutomatonAlternate}. This result makes use of a Büchi automaton called the zero automaton which generalizes that defined by Frougny~\cite{Frougny1992} and which is intimately linked with the Büchi automaton computing the normalization in alternate bases.

The paper is organized as follows. In Section~\ref{Sec : Preliminaries}, we first fix some notation and we provide the necessary background for a clear understanding of this work. Then, in Section~\ref{Sec : SpectrumComplex}, we study the spectrum $X^A(\delta)$ of a complex base $\pr$ over an arbitrary alphabet $A$ of complex numbers. In doing so, we also define and study the set $Z(\pr,A)$ of $\pr$-representations of zero over $A$ and an associated zero Büchi automaton $\mathcal{Z}(\pr,A)$. In Section~\ref{Sec : SpectrumAlternate}, we define the spectrum associated with an alternate base $\B=(\beta_0,\dots,\beta_{p-1})$ as a particular case of the complex spectra studied in the previous section. We then prove that the alternate base spectrum has no accumulation point if and only if the set of $\B$-representations of zero is accepted by a finite Büchi automaton, and furthermore, if and only if the alternate zero automaton is finite. Sections~\ref{Sec : NecessaryConditionForPeriodicity} and~\ref{Sec : SpectrumPeriodicity} are  concerned with the algebraic properties of the alternate base $\B=(\beta_0,\dots,\beta_{p-1})$ determining sofic systems. In particular, Section~\ref{Sec : NecessaryConditionForPeriodicity} and Section~\ref{Sec : SpectrumPeriodicity} respectively give the necessary and the sufficient conditions stated above. In Section~\ref{Sec : Normalization}, we show that if $\pr=\prod_{i=0}^{p-1}\beta_i$ is a Pisot number and each of the bases $\beta_i$ belongs to the algebraic field $\Q(\pr)$ then the normalization in the alternate base $\B$ is computable by a finite Büchi automaton, and we effectively construct such an automaton. We end this paper with some open questions.

\section{Preliminaries}
\label{Sec : Preliminaries}

Throughout this text, an interval of non-negative integers $\{i, \ldots , j\}$ with $i\le j$ is denoted $[\![i,j]\!]$ and $\floor{\cdot}$ and $\ceil{\cdot}$ respectively denote the floor and ceiling functions.

\subsection{Numbers}
An \emph{algebraic number} is a complex number that is a root of a monic polynomial (a polynomial whose leading coefficient is 1) with coefficients in $\Q$. The \emph{minimal polynomial} of an algebraic number $\beta$ is the monic polynomial of minimal degree that is annihilated by $\beta$. It is an irreducible polynomial over $\Q$ and its degree is the \emph{degree} of the algebraic number $\beta$. Roots of the same irreducible polynomial over $\Q$ are distinct and are said to be \emph{algebraically conjugate}.

An algebraic number is an \emph{algebraic integer} if it is a root of a monic polynomial in $\Z[x]$. It can be shown that its minimal polynomial has also integer coefficients, and thus a rational number which is not an integer is never an algebraic integer. The set of all algebraic integers is closed under addition, subtraction and multiplication and therefore is a commutative subring of the complex numbers.

A \emph{Pisot number} is an algebraic integer $\beta>1$ whose Galois conjugates all have modulus less than $1$. In particular, every integer is a Pisot number.

The smallest subfield of the field $\C$ of complex numbers containing a complex number $\beta$ is denoted by $\Q(\beta)$. If $\beta$ is an algebraic number of degree $d$ then this field is of the form
\[
	\Q(\beta)=\{\sum_{j=0}^{d-1} a_j\beta^j : a_j\in \Q\}.
\]
Let $\beta$ be an algebraic number of degree $d$ and let $\beta_2,\ldots,\beta_{d}$ be its Galois conjugates. We set $\beta_1=\beta$. Then for all $k\in [\![1,d]\!]$, the fields $\Q(\beta)$ and $\Q(\beta_k)$ are isomorphic by the isomorphism
\[
	\psi_k\colon \Q(\beta)\to \Q(\beta_k),\ 
	\sum_{j=0}^{d-1}a_j\beta^j \mapsto \sum_{j=0}^{d-1}a_j(\beta_k)^j.
\]
The norm $N_{\Q(\beta)/\Q}$ on the algebraic field $\Q(\beta)$ is defined by
\[
N_{\Q(\beta)/\Q}\colon \Q(\beta)\to \Q, \ x \mapsto \prod_{k=1}^d \psi_k(x).
\]
Whenever $x$ is an algebraic integer in $\Q(\beta)$, the rational $N_{\Q(\beta)/\Q}(x)$ is actually an integer.

\subsection{Words}
We make use of common notions in formal language theory. An \emph{alphabet} is a finite set of symbols, called \emph{letters}, or, in the context of numeration systems, \emph{digits}. A finite (resp.\ infinite) \emph{word} is a finite (resp.\ infinite) concatenation of letters. The length of a finite word is the number of its letters. The set of all finite words over an alphabet $A$ equipped with the binary operation of concatenation and the empty word as the neutral element is the monoid $A^*$. A subset of $A^*$ is called a \emph{language} over $A$. The set of infinite words over $A$ is denoted $A^{\N}$.

Throughout this text, the letters of a word $a$ are denoted by $a_n$ and are indexed from $0$, so that $a=a_0a_1a_2\cdots$ if $a$ is an infinite word and $a=a_0a_1a_2\cdots a_{\ell-1}$ if $a$ is a finite word of length $\ell$. If a word $a\in A^*\cup A^{\N}$ can be written as $a=uvw$ for some $u,v\in A^*$, $w\in A^*\cup A^{\N}$ then $v$ is a \emph{factor} and $u$ is a \emph{prefix} of the word $a$. The language of factors and prefixes of a given word $w$ are denoted by $\Fac(w)$ and $\Pref(w)$ respectively.

A finite word $w$ written as $k$-fold repetition of its factor $v$ is denoted $w=v^k=v\cdots v$. Similarly, infinite repetition of the factor $v$ is $v^\omega=vvv\cdots$. An infinite word $a$ of the form $a=uv^\omega$ for some $u,v\in A^*$ is called \emph{eventually periodic}. In the case where $u$ is the empty word, $a$ is said to be \emph{purely periodic}.

\subsection{Automata}

An \emph{automaton} is a  quintuplet $\mathcal{A}=(Q,I,F,A,E)$ where $Q$ is the set of \emph{states}, $I\subseteq Q$ is the set of \emph{initial states}, $F\subseteq Q$ is the set of \emph{final states}, $A$ is an alphabet, and $E\subseteq Q\times A\times Q$ is the set of \emph{transitions}. The automaton $\mathcal{A}$ is said to be \emph{deterministic} if the transition relation $E$ is actually a function from $Q\times A$ to $Q$, in which case we write $E\colon Q\times A\to Q$. The automaton $\mathcal{A}$ is said to be \emph{finite} if the set of states $Q$ is finite (since an alphabet is always considered to be finite in this text). A finite word $a_0\cdots a_{\ell-1}$ over the alphabet $A$ is \emph{accepted by $\mathcal{A}$} if it is the label of a \emph{successful path in $\mathcal{A}$}, i.e., if there exists $q_0,\ldots,q_{\ell}\in Q$ such that $q_0\in I$, $q_\ell\in F$ and for all $n\in[\![0,\ell-1]\!]$, we have $(q_n,a_n,q_{n+1})\in E$.

\emph{Büchi automata} are defined as classical automata except for the acceptance criterion which has to be adapted in order to deal with infinite words: an infinite word is accepted if it labels an initial path going infinitely many times through a final state. A major difference between Büchi and classical automata is that a set of infinite words accepted by a finite Büchi automaton is not necessarily accepted by a deterministic one. For more on Büchi automata, we refer the reader to~\cite{PerrinPin2004}.

\subsection{Spectrum and set of $\beta$-representations of zero}

For a real number $\beta>1$ and $d\in\N$, 
we let $Z(\beta,d)$ denote the set of $\beta$-representations of zero over the alphabet $[\![-d,d]\!]$:
\[
	Z(\beta,d)
	=\{a \in [\![-d,d]\!]^{\N} : 
	\sum_{n=0}^{+\infty} \frac{a_n}{\beta^{n+1}}=0\}.
\]
The \emph{$d$-spectrum of $\beta$} is the set
\[
	X^d(\beta)
	=\{\sum_{n=0}^{\ell-1}a_n\beta^{\ell-1-n} : \ell\in\N,\ a_0,a_1,\ldots,a_{\ell-1}\in [\![-d,d]\!]\}.
\]
The following theorem linking these two sets was proved in~\cite{FrougnyPelantova2018}.

\begin{theorem}[Frougny and Pelantov\'a~\cite{FrougnyPelantova2018}]
\label{Thm : SetZeroAndSpectrumBeta}
Let $\beta>1$ and $d\in \N$. Then $Z(\beta,d)$ is accepted by a finite Büchi automaton if and only if the spectrum $X^d(\beta)$ has no accumulation point in $\R$.
\end{theorem}

The next result is due to Akiyama and Komornik~\cite{AkiyamaKomornik2013} and Feng~\cite{Feng2016}.

\begin{theorem}[Akiyama and Komornik~\cite{AkiyamaKomornik2013}, Feng~\cite{Feng2016}]
\label{Thm : AccPointAkiyamaKomornikFeng}
Let $\beta>1$ and $d\in\N$.
The spectrum $X^d(\beta)$ has an accumulation point in $\R$ if and only if $\beta-1<d$ and $\beta$ is not a Pisot number.
\end{theorem}

The following result is a direct consequence of  Theorems~\ref{Thm : SetZeroAndSpectrumBeta} and~\ref{Thm : AccPointAkiyamaKomornikFeng}, as noticed in~\cite{FrougnyPelantova2018}.

\begin{theorem}
\label{Thm : EquivalencesZeroReprFrougnyPelantova}
Let $\beta>1$. The following assertions are equivalent.
\begin{enumerate}
\item The set $Z(\beta,d)$ is accepted by a finite Büchi automaton for all $d\ge 0$.
\item The set $Z(\beta,d)$ is accepted by a finite Büchi automaton for one $d\ge \ceil{\beta}-1$.
\item $\beta$ is a Pisot number.
\end{enumerate}
\end{theorem}

\subsection{Alternate bases}

Let $p$ be a positive integer and $\B=(\beta_0,\ldots,\beta_{p-1})$ be a $p$-tuple of real numbers greater than $1$. Such a $p$-tuple $\B$ is called an \emph{alternate base} and $p$ is called its \emph{length}. A \emph{$\B$-representation} of a real number $x$ is an infinite sequence $a=a_0a_1a_2\cdots$ of integers such that 
\begin{equation}
\label{Eq : valueAlternatBase}
	x=\sum_{m=0}^{+\infty} \sum_{i=0}^{p-1} \frac{a_{mp+i}}{(\beta_0\cdots\beta_{p-1})^m\beta_0\cdots\beta_i}.
\end{equation}
We use the convention that for all $n\in \Z$, $\beta_n=\beta_{n \bmod p}$ and 
\[\B^{(n)}=(\beta_n,\ldots,\beta_{n+p-1}).\]
Therefore, the equality~\eqref{Eq : valueAlternatBase} can be rewritten as
\[
	x=\sum_{n=0}^{+\infty}\frac{a_n}{\prod_{k=0}^n\beta_k}.
\]
In order to simplify notation, we sometimes write $\val_{\B}(a)$ in order to designate the second term of the previous equality. 

For $x\in[0,1]$, a distinguished $\B$-representation $\varepsilon_0(x)\varepsilon_1(x)\varepsilon_2(x)\cdots$, called the \emph{$\B$-expansion} of $x$, is obtained from the \emph{greedy algorithm}:
\begin{itemize}
\item $\varepsilon_0(x)=\floor{\beta_0 x}$ and $r_0(x)=\beta_0 x-\varepsilon_0(x)$
\item $\varepsilon_n(x)=\floor{\beta_n r_{n-1}(x)}$ and $r_n=\beta_n r_{n-1}(x)-\varepsilon_n(x)$ for $n\in\N_{\ge 1}$.
\end{itemize}
The $\B$-expansion of $x$ is denoted $\DB(x)$. Clearly, the $n$-th digit $\varepsilon_n(x)$ takes value in $\{0,1,\dots,\lfloor\beta_n\rfloor\}$. The numbers $r_n(x)$ computed by the greedy algorithm are called \emph{remainders}. The \emph{quasi-greedy $\B$-expansion of $1$}, denoted $d_{\B}^*(1)$, is the $\B$-representation of $1$ given by $\lim_{x\to 1}\DB(x)$ (where the limit is taken with respect to the product topology on infinite words).

The following theorem characterizing the admissible sequences is an analogue of Parry's theorem~\cite{Parry1960}.

\begin{theorem}[Charlier and Cisternino \cite{CharlierCisternino2021}]
\label{thm:Parry}
An infinite word $a$ over $\N$ is the $\B$-expansion of some real number in the interval $[0,1)$ if and only if $a_na_{n+1}\cdots<_{\lex} \DBi{n}^*(1)$ for all $n\in\N$.
\end{theorem}

The $\B$-shift associated with an alternate base $\B$ is defined as the topological closure (with respect to the product topology of infinite words) of the union $\bigcup_{i=0}^{p-1}\{\DBi{i}(x) \colon x\in [0,1)\}$. Generalizing a celebrated result of Bertrand-Mathis for real bases \cite{Bertrand-Mathis1986}, Theorem~\ref{thm:Parry} was used in \cite{CharlierCisternino2021} for showing that the $\B$-shift is sofic if and only if $\qDBi{i}(1)$ is eventually periodic for all $i\in \Int$. We refer to such bases as the \emph{Parry alternate bases}.

\section{Spectrum and representations of zero in complex bases}
\label{Sec : SpectrumComplex}

In this section, we generalize Theorem~\ref{Thm : SetZeroAndSpectrumBeta} to the context of complex bases and general alphabets of complex digits. Here, a complex base is a complex number $\delta$ such that $|\delta|>1$. For a complex base $\delta$ and an alphabet $A$ of complex numbers, we define the set of $\delta$-representations of zero over $A$ by 
\[
	Z(\delta,A)=
	\{a\in A^{\N} : \sum_{n=0}^{+\infty} \frac{a_n}{\delta^{n+1}}=0\}
\]
and the \emph{spectrum of $\delta$ over the alphabet $A$} by
\begin{equation}
\label{eq:spectrum}
	X^A(\delta)
	=\{\sum_{n=0}^{\ell-1}a_n\delta^{\ell-1-n} : \ell\in\N,\ a_0,a_1,\ldots,a_{\ell-1}\in A\}.
\end{equation}
We say that a word $a_0\cdots a_{\ell-1}$ over $A$ \emph{corresponds} to the element $\sum_{i=0}^{\ell-1}a_i\delta^{\ell-1-i}$ in the spectrum $X^A(\delta)$. For the remaining part of this section, we consider a fixed complex base $\delta$ and a fixed alphabet $A\subseteq\C$. 

We also define a relation $\sim_{\delta,A}$ over $A^*$: for $a,b \in A^*$, we set $a\sim_{\delta,A} b$ whenever for all $s\in A^{\N}$, we have $as\in Z(\delta,A) \iff bs\in Z(\delta,A)$. Obviously, this is a right congruence over $A^*$ and $A^*\setminus \Pref(Z(\delta,A))$ is one of its equivalence classes. In the context of real bases $\beta$ and integer digits, this right congruence may be interpreted in terms of the remainders of the Euclidean division of polynomials in $\Z[x]$ by $x-\beta$; see~\cite{Frougny1992}. This interpretation is no longer possible in the present context of complex digits.

\begin{lemma}
\label{Lem : ComplexEquivalenceIIFEquality}
Let $a,b\in \Pref(Z(\delta,A))$ be such that $|a|=k$ and $|b|=\ell$. We have $a\sim_{\delta,A} b$ if and only if 
\[
	\sum_{n=0}^{k-1}a_n\delta^{k-1-n}
	=\sum_{n=0}^{\ell-1}b_n\delta^{\ell-1-n},
\]
that is, the words $a$ and $b$ correspond to the same element in the spectrum $X^A(\delta)$.
\end{lemma}

\begin{proof}
Suppose that $a\sim_{\delta,A} b$. Since $a$ and $b$ belong to the set $\Pref(Z(\delta,A))$, there exists $s\in A^{\N}$ such that $as,bs\in Z(\delta,A)$. We get 
\[
	-\sum_{n=0}^{+\infty} \frac{s_n}{\delta^{n+1}}
	=\sum_{n=0}^{k-1}a_n\delta^{k-1-n}
	=\sum_{n=0}^{\ell-1}b_n\delta^{\ell-1-n}.
\]
Conversely, suppose that $a\not\sim_{\delta,A} b$. Without loss of generality we can suppose that $as\in Z(\delta,A)$ and $bs\notin Z(\delta,A)$ for some $s\in A^{\N}$. Then
\[
	-\sum_{n=0}^{+\infty} \frac{s_n}{\delta^{n+1}}
	=\sum_{n=0}^{k-1}a_n\delta^{k-1-n}
	\ne \sum_{n=0}^{\ell-1}b_n\delta^{\ell-1-n}. 
\]
\end{proof}

\begin{lemma}
\label{Lem : NoAccEquivalenceInfinite}
If the spectrum $X^A(\delta)$ has an accumulation point in $\C$ then there exists an infinite word in $Z(\delta,A)$ with pairwise non equivalent prefixes with respect to the right congruence $\sim_{\delta,A}$. In particular, the right congruence $\sim_{\delta,A}$ has infinitely many classes.
\end{lemma}

\begin{proof}
Suppose that the spectrum $X^A(\delta)$ has a complex accumulation point. Then there exists an injective sequence $(x^{(j)})_{j\in \N}$ in $X^A(\delta)$ such that $\lim_{j\to +\infty} x^{(j)}$ is finite. For each $j\in\N$ we let $\rho(j)$ denote the minimal exponent such that there exists a representation of $x^{(j)}$ in the form $x^{(j)}_0\cdots x^{(j)}_{\rho(j)-1}\in A^*$, that is 
\[
	x^{(j)}
	=\sum_{n=0}^{\rho(j)-1} x_n^{(j)}\delta^{\rho(j)-1-n}.
\]
Obviously, the sequence $(\rho(j))_{j\in \N}$ is unbounded, and  without loss of generality we can assume that $(\rho(j))_{j\in \N}$ is strictly increasing. Thus $\lim_{j\to +\infty}\rho(j)=+\infty$ and we get
\begin{equation}
\label{Eq : Lim}
	\lim_{j\to +\infty} \frac{x^{(j)}}{\delta^{\rho(j)}} 
	=\lim_{j\to +\infty} \sum_{n=0}^{\rho(j)-1}\frac{x_n^{(j)}}{\delta^{n+1}} 
	=0.
  \end{equation}
With this, we will show the existence of the desired $\delta$-representation $a$ of zero over $A$. Set $a_0$ as a digit in $A$ which occurs infinitely many times among $x_0^{(j)}$ with $j\in\N$.
Inductively, for $n\geq 1$, set $a_n$ as a digit in $A$ which occurs infinitely many times among $x_n^{(j)}$, where $j\in\N$ runs through the indices such that $x_0^{(j)}\cdots x_{n-1}^{(j)}=a_0\cdots a_{n-1}$. By~\eqref{Eq : Lim}, we get that
\[
	\sum_{n=0}^{+\infty}\frac{a_n}{\delta^{n+1}}=0,
\] 
that is, that $a=a_0a_1a_2\cdots$ belongs to the set $Z(\delta ,A)$.
  
We will show that no pair of distinct prefixes of the infinite word $a$ belong to the same equivalence class. To show this by contradiction, we consider $k,\ell\in \N$ such that $a_0\cdots a_{k-1}\sim_{\delta,A} a_0\cdots a_{\ell-1}$ with $k>\ell$. By construction, there exists $j\in \N$ such that $a_0\cdots a_{k-1}$ is a prefix of $x^{(j)}_0\cdots x^{(j)}_{\rho(j)-1}$. Moreover, by Lemma~\ref{Lem : ComplexEquivalenceIIFEquality}, we get
\begin{align*}
	x^{(j)}
	&=\sum_{n=0}^{k-1} a_n\delta^{\rho(j)-1-n} 
	+\sum_{n=k}^{\rho(j)-1} x^{(j)}_n\delta^{\rho(j)-1-n}\\
	&=\delta^{\rho(j)-k}\sum_{n=0}^{k-1} a_n\delta^{k-1-n} 
	+\sum_{n=k}^{\rho(j)-1} x^{(j)}_n\delta^{\rho(j)-1-n}\\
	&=\delta^{\rho(j)-k}\sum_{n=0}^{\ell-1} a_n\delta^{\ell-1-n} 
	+\sum_{n=k}^{\rho(j)-1} x^{(j)}_n\delta^{\rho(j)-1-n}\\
	&=\sum_{n=0}^{\ell-1} a_n\delta^{\rho(j)-k+\ell-1-n} 
	+\sum_{n=\ell}^{\rho(j)-k+\ell-1} x^{(j)}_{n+k-\ell}\delta^{\rho(j)-k+\ell-1-n}.
\end{align*}
Thus, we have found a representation $a_0\cdots a_{\ell-1} x^{(j)}_k\cdots x^{(j)}_{\rho(j)-1}$ of $x^{(j)}$ which is shorter than $x^{(j)}_0\cdots x^{(j)}_{\rho(j)-1}$. This contradicts the definition of $\rho(j)$. 
\end{proof}

Similarly as what is done in~\cite{Frougny1992}, we define a Büchi automaton, which we call the \emph{zero automaton~$\mathcal{Z}(\delta,A)=(Q,0,Q,A,E)$}. The set of states is $Q=X^A(\delta)\cap \{z\in\C : |z|\le \frac{M}{|\delta|-1}\}$ where $M=\max\{|a| : a\in A\}$, and the transitions are given by the triplets $(z,a,z\delta+a)$ in $Q\times A\times Q$.

\begin{proposition}
\label{Pro : ZeroAutoAcceptsReprOf0Complex}
The zero automaton $\mathcal{Z}(\delta,A)$ accepts the set $Z(\delta,A)$.
\end{proposition}

\begin{proof}
Let $a$ be an infinite word accepted by $\mathcal{Z}(\delta,A)$. For each $\ell\in \N$, the prefix $a_0\cdots a_{\ell-1}$ labels a path in $\mathcal{Z}(\delta,A)$ from the initial state $0$ to the state $\sum_{n=0}^{\ell-1}a_n\delta^{\ell-1-n}$, i.e., its corresponding element in the spectrum $X^A(\delta)$. By definition of the set of states $Q$, we get that the sequence $(\sum_{n=0}^{\ell-1}a_n\delta^{\ell-1-n})_{\ell\in \N}$ is bounded. Hence, we obtain that
\[
	\sum_{n=0}^{+\infty} \frac{a_n}{\delta^{n+1}}
	=\lim_{\ell\to +\infty} 
	\frac{\sum_{n=0}^{\ell-1}a_n\delta^{\ell-1-n}}{\delta^\ell}
	=0.
\]

Conversely, consider an infinite word $a$ over $A$ that is not accepted by $\mathcal{Z}(\delta,A)$. Then there exists $\ell\in\N$ such that $|\sum_{n=0}^{\ell-1}a_n\delta^{\ell-1-n}|>\frac{M}{|\delta|-1}$. Then
\[
	\left|\sum_{n=0}^{+\infty} \frac{a_n}{\delta^{n+1}}\right|
	\ge \left|\sum_{n=0}^{\ell-1}
	\frac{a_n}{\delta^{n+1}}\right|
	-\sum_{n=\ell}^{+\infty}\frac{M}{|\delta|^{n+1}}
	=\frac{\left|\sum_{n=0}^{\ell-1}a_n\delta^{\ell-n-1}\right|-\frac{M}{|\delta|-1}}{|\delta|^\ell}
	>0. 
\]
\end{proof}

We are now ready to state and prove the main theorem of this section, which is a generalization of Theorem~\ref{Thm : SetZeroAndSpectrumBeta}. This solves a problem that was left open in~\cite{FrougnyPelantova2018}.

\begin{theorem}
\label{Thm : SetZeroSpectrumAccZeroAutomatonComplex}
Let $\delta$ be a complex number such that $|\delta|>1$ and let $A$ be an alphabet of complex numbers. Then the following assertions are equivalent.
\begin{enumerate}
\item The set $Z(\delta,A)$ is accepted by a finite Büchi automaton.
\item The right congruence $\sim_{\delta,A}$ has finitely many classes. 
\item The spectrum $X^A(\delta)$ has no accumulation point in $\C$.
\item The zero automaton $\mathcal{Z}(\delta,A)$ is finite.
\end{enumerate} 
\end{theorem}

\begin{proof}
Suppose that $Z(\delta,A)$ is accepted by a finite Büchi automaton. Similarly as in~\cite[Lemma 3.2]{Frougny1992}), we get that the right congruence $\sim_{\delta,A}$ has only finitely many classes. Hence $(1)\implies (2)$. The implication $(2)\implies (3)$ is given by Lemma~\ref{Lem : NoAccEquivalenceInfinite}. The implication $(3)\implies (4)$ follows directly from the definition of the zero automaton. Finally, the implication $(4)\implies (1)$ follows from Proposition~\ref{Pro : ZeroAutoAcceptsReprOf0Complex}.
\end{proof}

Note that the zero automaton is deterministic. Therefore, the previous result shows in particular that if the set $Z(\delta,A)$ is accepted by an arbitrary Büchi automaton, possibly non deterministic, then it must be also accepted by a deterministic one.

\section{Spectrum and representations of zero in alternate bases}
\label{Sec : SpectrumAlternate}

From now on, we consider a fixed  positive integer $p$ and an alternate base $\B=(\beta_0,\ldots,\beta_{p-1})$, and we set $\pr=\prod_{i=0}^{p-1}\beta_i$. Moreover, we consider a $p$-tuple $\boldsymbol{D}=(D_0,\ldots,D_{p-1})$ where, for all $i\in \Int$, $D_i$ is an alphabet of integers containing $0$. We use the convention that for all $n\in \Z$, $D_n=D_{n\bmod p}$ and $\boldsymbol{D}^{(n)}=(D_n,\ldots, D_{n+p-1})$. 

Grouping terms $p$ by $p$, Equality~\eqref{Eq : valueAlternatBase} can be written as 
\[
	x=\sum_{m=0}^{+\infty} \frac{\sum_{i=0}^{p-1}a_{mp+i}\beta_{i+1}\cdots\beta_{p-1}}{\delta^{m+1}}.
\]
If we add the constraint that each letter $a_n$ belongs to $D_n$, then we obtain a $\pr$-representation of $x$ over the alphabet
\begin{equation}
\label{Eq : Digits}	
	\Dig
	=\{\sum_{i=0}^{p-1} a_i\beta_{i+1}\cdots\beta_{p-1} : 
	\forall i\in\Int,\ a_i\in D_i\}.
\end{equation}
An analogous link was done in~\cite{CharlierCisterninoDajani2021} in the restricted case of alphabets $D_i=[\![0,\ceil{\beta_i}-1]\!]$.

We let $\otimes_{n=0}^{+\infty}D_n$ denote the set of infinite words over the alphabet $\cup_{i=0}^{p-1}D_i$ such that for all $n\in\N$, the $n$-th letter belongs to the alphabet $D_n$:
\[
	\bigotimes_{n=0}^{+\infty}D_n
	=\{a\in (\cup_{i=0}^{p-1}D_i)^\N :
	\forall n\in\N,\ a_n\in D_n\}.
\]
We let $Z(\B,\boldsymbol{D})$ denote the set of $\B$-representations of zero the $n$-th digit of which belongs to the alphabet $D_n$:
\[
	Z(\B,\boldsymbol{D})
	=\{a \in \bigotimes_{n=0}^{+\infty} D_n : \sum_{n=0}^{+\infty}\frac{a_n}{\prod_{k=0}^n\beta_k}=0\}.
\]
This set can be seen as a set of infinite words over the alphabet $\cup_{i=0}^{p-1}D_i$.

For $\pr=\prod_{i=0}^{p-1}\beta_i$ and the alphabet $\Dig$, the spectrum $X^{\Dig}(\pr)$ defined in~\eqref{eq:spectrum} can be rewritten as
\[
	X^{\Dig}(\pr) 
	= X^{D_0}(\pr)\cdot \beta_1\cdots\beta_{p-1} 
	+ X^{D_1}(\pr)\cdot \beta_2\cdots \beta_{p-1} 
	+ \cdots 
	+ X^{D_{p-1}}(\pr).
\]
For the sake of simplicity, for each $i\in \Int$, we let $X(i)$ denote the spectrum built from the shifted base $\B^{(i)}$ and the shifted $p$-tuple of alphabets $\boldsymbol{D}^{(i)}$. In particular, we have $X(0)=X^{\Dig}(\pr)$.

\begin{lemma}
\label{Lem : SpectrumAccAllSpectraAcc}
For each $i\in \Int$, we have $X(i)\cdot \beta_i+D_i=X(i+1)$ where it is understood that $X(p)=X(0)$.
\end{lemma}

\begin{proof}
For each $i\in \Int$, we have
\[
	X(i) 
	= X^{D_i}(\pr)\cdot \beta_{i+1}\cdots\beta_{i+p-1}  
	+ X^{D_{i+1}}(\pr)\cdot \beta_{i+2}\cdots \beta_{i+p-1}  
	+ \cdots 
	+ X^{D_{i+p-1}}(\pr).
\]
Since $(X^{D_i}(\pr)\cdot\beta_{i+1}\cdots\beta_{i+p-1})\cdot \beta_i+D_i
=X^{D_i}(\pr)\cdot \delta +D_i
= X^{D_i}(\pr)$, the conclusion follows.
\end{proof}

\begin{lemma}
\label{Lem : SpectrumWordElementCorrespondence}
For all $\ell\in\N$, we have $\sum_{n=0}^{\ell-1} D_n \cdot \beta_{n+1}\cdots \beta_{\ell-1} \subset X(\ell\bmod p)$. 
\end{lemma}

\begin{proof}
This is an easy induction using Lemma~\ref{Lem : SpectrumAccAllSpectraAcc} and the fact that $D_n\subset X(n+1)$ for all $n\in\N$.
\end{proof}

In view of the previous lemma, if for each $n\in[\![0,\ell-1]\!]$, the digit $a_n$ belongs to the alphabet $D_n$, then we say that the finite word $a_0\ldots a_{\ell-1}$ \emph{corresponds} to the element $\sum_{n=0}^{\ell-1} a_n \beta_{n+1}\cdots \beta_{\ell-1}$ of the spectrum $X(\ell\bmod p)$. 

Let us now generalize the notion of zero automaton to the context of alternate bases. For each $i\in\Int$, we define
\[
	M^{(i)}=\sum_{n=i}^{+\infty}\frac{\max(D_n)}{\prod_{k=i}^n\beta_k}
\qquad
\text{ and }
\qquad	
	m^{(i)}=\sum_{n=i}^{+\infty}\frac{\min(D_n)}{\prod_{k=i}^n\beta_k}
\] 
where $\max(D_n)$ and $\min(D_n)$ respectively denote the maximal and minimal digit in the alphabet $D_n$. As usual, for $n\in \Z$, we set $M^{(n)}=M^{(n\bmod p)}$ and $m^{(n)}=m^{(n\bmod p)}$. We define the \emph{zero automaton} $\mathcal{Z}(\B,\boldsymbol{D})$ associated with the alternate base $\B$ and the $p$-tuple of alphabets $\boldsymbol{D}$ as
\[
	\mathcal{Z}(\B,\boldsymbol{D})
	=(Q_{\B,\boldsymbol{D}}, (0,0), Q_{\B,\boldsymbol{D}}, \cup_{i=0}^{p-1}D_i, E)
\]
where 
\begin{itemize}
\item $Q_{\B,\boldsymbol{D}} = \bigcup_{i=0}^{p-1} (\{i\} \times ( X(i) \cap [-M^{(i)},-m^{(i)}]))$
\item $E$ is the set of transitions defined as follows: for $(i,s),(j,t)\in Q_{\B,\boldsymbol{D}}$ and $a\in \cup_{i=0}^{p-1}D_i$, there is a transition $(i,s) \xrightarrow[ \quad ]{a} (j,t)$ if and only if $j\equiv i+1 \pmod p$, $a\in D_i$ and $t=\beta_i s +a$.
\end{itemize}
Observe that since we have assumed that all the alphabets $D_i$ contain the digit $0$, the initial state $(0,0)$ is indeed an element of $Q_{\B,\boldsymbol{D}}$. Moreover, if $s\in X(i)$ and $a\in D_i$ then $\beta_i s+a\in X(i+1)$ by Lemma~\ref{Lem : SpectrumAccAllSpectraAcc}.

\begin{proposition}
\label{Pro : ZeroAutoAcceptsReprOf0}
The zero automaton $\mathcal{Z}(\B,\boldsymbol{D})$ accepts the set $Z(\B,\boldsymbol{D})$.
\end{proposition}

\begin{proof}
Let $a$ be an infinite word accepted by $\mathcal{Z}(\B,\boldsymbol{D})$. For each $\ell\in \N$, the prefix $a_0\cdots a_{\ell-1}$ labels a path in $\mathcal{Z}(\B,\boldsymbol{D})$ from the initial state $(0,0)$ to the state 
\[
 	(\ell \bmod p, \sum_{n=0}^{\ell-1}a_n\beta_{n+1}\cdots\beta_{\ell-1}).
\]
Therefore, the sequence $\big(\sum_{n=0}^{\ell-1}a_n\beta_{n+1}\cdots\beta_{\ell-1}\big)_{\ell\in \N}$ is bounded. Hence, we get 
\[
	\sum_{n=0}^{+\infty}\frac{a_n}{\prod_{k=0}^n\beta_k}
	=\lim_{\ell\to +\infty} \frac{\sum_{n=0}^{\ell-1}a_n\beta_{n+1}\cdots\beta_{\ell-1}}{\prod_{n=0}^{\ell-1}\beta_n}
	=0.
\]

Conversely, consider an infinite word $a$ such that $a_n\in D_n$ for all $n\in\N$ and that is not accepted by $\mathcal{Z}(\B,\boldsymbol{D})$. Then, there exists $\ell\in\N$ such that $(\ell \bmod p, \sum_{n=0}^{\ell-1}a_n\beta_{n+1}\cdots\beta_{\ell-1})\notin Q_{\B,\boldsymbol{D}}$. In view of Lemma~\ref{Lem : SpectrumWordElementCorrespondence}, we get $\sum_{n=0}^{\ell-1}a_n\beta_{n+1}\cdots\beta_{\ell-1}\notin [-M^{(\ell)},-m^{(\ell)}]$. Suppose that $\sum_{n=0}^{\ell-1}a_n\beta_{n+1}\cdots\beta_{\ell-1}>-m^{(\ell)}$ (the other case is symmetric). Then
\[
	\sum_{n=0}^{+\infty}\frac{a_n}{\prod_{k=0}^n\beta_k}
	\ge \sum_{n=0}^{\ell-1}
	\frac{a_n}{\prod_{k=0}^n\beta_k}
	+\sum_{n=\ell}^{+\infty}\frac{\min(D_n)}{\prod_{k=0}^n\beta_k}
	=\frac{\sum_{n=0}^{\ell-1}a_n\beta_{n+1}\cdots\beta_{\ell-1}+m^{(\ell)}}{\prod_{n=0}^{\ell-1}\beta_n}
	>0. 
\]
\end{proof}

\begin{example}
\label{Ex : BaseHabituelle}
Consider the alternate base $\B=(\frac{1+ \sqrt{13}}{2},\frac{5+ \sqrt{13}}{6})$ and the pair of alphabets $\boldsymbol{D}=([\![-2,2]\!],[\![-1,1]\!])$. Then $M^{(0)}=\val_{\B}((21)^\omega)\simeq 1.67994$ and $M^{(1)}= \val_{\B^{(1)}}((12)^\omega) \simeq 1.86852$. The zero automaton $\mathcal{Z}(\B,\boldsymbol{D})$ is depicted in Figure~\ref{Fig : ZeroAutomaton-1+Sqrt13-accessible} where the states with first components $0$ and $1$ are colored in pink and purple respectively, and where the edges labeled by $-2,-1,0,1$ and $2$ are colored in dark blue, dark green, red, light green and light blue respectively. For instance, the infinite words $1(\overline{1}0)^\omega$ and $(0\overline{1}21\overline{2}\overline{1})^\omega$ have value $0$ in base $\B$ (where $\overline{1}$ and $\overline{2}$ designate the digits $-1$ and $-2$ respectively).
\begin{figure}
\centering
\begin{tikzpicture}[node distance=7 cm, scale = 1.1, shape=rectangle,transform shape,  every state/.style = {draw =black}, inner sep=3pt]
\tikzstyle{every node}=[shape=rectangle, fill=none, draw=black,
 inner sep=2pt]
\node[accepting,fill=pink] (00) at (0,0) {\small{$0$}};
\node[accepting,fill=pink] (01) at (2,-3) {\small{$1$}};
\node[accepting,fill=pink] (0m1) at (-2,-3) {\small{${-}1$}};
\node[accepting,fill=pink] (0mb1) at (-4,0) {\small{${-}\beta_1$}};
\node[accepting,fill=pink] (0b1) at (4,0) {\small{$\beta_1$}};
\node[accepting,fill=pink] (0b1m1) at (6,0) {\small{$\beta_1{-}1$}};
\node[accepting,fill=pink] (0mb1p1) at (-6,0) {\small{${-}\beta_1{+}1$}};
\node[accepting,fill=pink] (02b1m2) at (2,-6) {\small{$2\beta_1{-}2$}};
\node[accepting,fill=pink] (0m2b1p2) at (-2,-6) {\small{${-}2\beta_1{+}2$}};
\node[accepting,fill=pink] (0b1m2) at (6,-6) {\small{$\beta_1{-}2$}};
\node[accepting,fill=pink] (0mb1p2) at (-6,-6) {\small{${-}\beta_1{+}2$}};
\node[accepting,fill=violet!50] (10) at (0,-3) {\small{$0$}};
\node[accepting,fill=violet!50] (11) at (2,0) {\small{$1$}};
\node[accepting,fill=violet!50] (1m1) at (-2,0) {\small{${-}1$}};
\node[accepting,fill=violet!50] (1b0m1) at (4,-3) {\small{$\beta_0{-}1$}};
\node[accepting,fill=violet!50] (1mb0p1) at (-4,-3) {\small{${-}\beta_0{+}1$}};
\node[accepting,fill=violet!50] (1b0m2) at (6,-3) {\small{$\beta_0{-}2$}};
\node[accepting,fill=violet!50] (1mb0p2) at (-6,-3) {\small{${-}\beta_0{+}2$}};
\node[accepting,fill=violet!50] (1b0m3) at (-4,-6) {\small{$\beta_0{-}3$}};
\node[accepting,fill=violet!50] (1mb0p3) at (4,-6) {\small{${-}\beta_0{+}3$}};

\tikzstyle{every node}=[shape=circle, fill=none, draw=black, minimum size=15pt, inner sep=2pt]
\tikzstyle{every path}=[color=black, line width=0.5 pt]
\tikzstyle{every node}=[shape=circle, minimum size=5pt, inner sep=2pt]
\draw [-Latex] (0,0.6) to node [above] {$ $} (00); 

\draw [-Latex,MyGreen] (00) to node {$ $} (1m1); 
\draw [-Latex,MyGreen!35!lime] (00) to node {$ $} (11); 
\draw [-Latex,red] (00) to [bend right=20] node {$ $} (10); 

\draw [-Latex,MyGreen] (10) to node {$ $} (0m1); 
\draw [-Latex,MyGreen!35!lime] (10) to node {$ $} (01); 
\draw [-Latex,red] (10) to [bend right=20] node {$ $} (00); 

\draw [-Latex,red] (11) to node {$ $} (0b1); 
\draw [-Latex,MyGreen] (11) to [bend left=20] node {$ $} (0b1m1); 

\draw [-Latex,red] (1m1) to node {$ $} (0mb1); 
\draw [-Latex,MyGreen!35!lime] (1m1) to [bend right=20] node {$ $} (0mb1p1); 

\draw [-Latex,MyGreen!35!lime] (0m1) to node {$ $} (1mb0p1); 
\draw [-Latex,cyan] (0m1) to [bend left=20] node {$ $} (1mb0p2); 

\draw [-Latex,MyGreen] (01) to node {$ $} (1b0m1); 
\draw [-Latex,blue] (01) to [bend right=20] node {$ $} (1b0m2); 

\draw [-Latex,cyan] (0mb1) to node {$ $} (1mb0p1); 

\draw [-Latex,blue] (0b1) to node {$ $} (1b0m1); 

\draw [-Latex,red] (0mb1p1) to [bend right=20] node {$ $} (1m1); 
\draw [-Latex,MyGreen!35!lime] (0mb1p1) to node {$ $} (10); 
\draw [-Latex,cyan] (0mb1p1) to  [out=30, in=150] node {$ $} (11); 

\draw [-Latex,red] (0b1m1) to [bend left=20] node {$ $} (11); 
\draw [-Latex,MyGreen] (0b1m1) to node {$ $} (10); 
\draw [-Latex,blue] (0b1m1) to  [out=150, in=30] node {$ $} (1m1); 

\draw [-Latex,MyGreen!35!lime] (1mb0p1) to node {$ $} (0m2b1p2); 

\draw [-Latex,MyGreen] (1b0m1) to node {$ $} (02b1m2); 

\draw [-Latex,MyGreen!35!lime] (0m2b1p2) to [bend left=20]  node {$ $} (1m1); 
\draw [-Latex,cyan] (0m2b1p2) to node {$ $} (10); 

\draw [-Latex,MyGreen] (02b1m2) to [bend right=20]  node {$ $} (11); 
\draw [-Latex,blue] (02b1m2) to node {$ $} (10); 

\draw [-Latex,MyGreen!35!lime] (1mb0p2) to node {$ $} (0mb1p2); 
\draw [-Latex,red] (1mb0p2) to node {$ $} (0mb1p1); 
\draw [-Latex,MyGreen] (1mb0p2) to node {$ $} (0mb1); 

\draw [-Latex,MyGreen] (1b0m2) to node {$ $} (0b1m2); 
\draw [-Latex,red] (1b0m2) to node {$ $} (0b1m1); 
\draw [-Latex,MyGreen!35!lime] (1b0m2) to node {$ $} (0b1); 

\draw [-Latex,red] (0b1m2) to node {$ $} (1mb0p1); 
\draw [-Latex,MyGreen!35!lime] (0b1m2)  to  node  {$ $} (-5.9,-3.3); 
\draw [-Latex,cyan] (0b1m2)  to node {$ $} (1mb0p3); 

\draw [-Latex,red] (0mb1p2) to node {$ $} (1b0m1); 
\draw [-Latex,MyGreen] (0mb1p2)  to  node  {$ $} (5.9,-3.3); 
\draw [-Latex,blue] (0mb1p2)  to node {$ $} (1b0m3); 

\draw [-Latex,red] (1b0m3)  to node {$ $} (01); 
\draw [-Latex,MyGreen] (1b0m3)  to [bend left=20]  node {$ $} (00); 

\draw [-Latex,red] (1mb0p3)  to node {$ $} (0m1); 
\draw [-Latex,MyGreen!35!lime] (1mb0p3)  to [bend right=20]  node {$ $} (00); 

\end{tikzpicture}
\caption{The zero automaton $\mathcal{Z}(\B,\boldsymbol{D})$ for $\B=(\frac{1+ \sqrt{13}}{2},\frac{5+ \sqrt{13}}{6})$ and $\boldsymbol{D}=([\![-2,2]\!],[\![-1,1]\!])$. The conventions for colors are described within Example~\ref{Ex : BaseHabituelle}.}
\label{Fig : ZeroAutomaton-1+Sqrt13-accessible}
\end{figure}
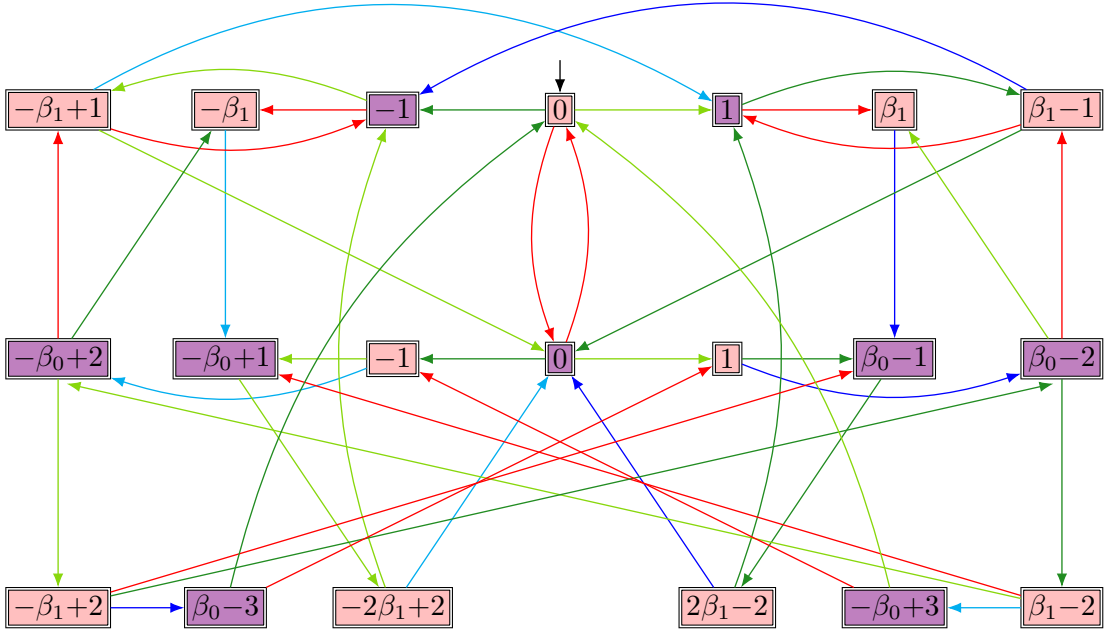
\end{example}

\begin{theorem}
\label{Thm : SetZeroSpectrumAccZeroAutomatonAlternate}
Let $\B$ be an alternate base of length $p$ and let $\boldsymbol{D}$ be a $p$-tuple of alphabets of integers containing $0$. Then the following assertions are equivalent.
\begin{enumerate}
\item The set $Z(\B,\boldsymbol{D})$ is accepted by a finite Büchi automaton.
\item The spectrum $X^{\Dig}(\pr)$ has no accumulation point in $\R$.
\item The zero automaton $\mathcal{Z}(\B,\boldsymbol{D})$ is finite.
\end{enumerate}
\end{theorem}

\begin{proof} 
By Lemma~\ref{Lem : SpectrumAccAllSpectraAcc}, if the spectrum $X^{\Dig}(\pr)$ has no accumulation point in $\R$ then for all $i\in\Int$, the spectrum $X(i)$ based on the cyclic shift $\B^{(i)}$ of the base and the corresponding shifted $p$-tuple of alphabets $\boldsymbol{D}^{(i)}$ has no accumulation point in $\R$ either. The implication $(2)\implies (3)$ then follows directly from the definition of the set of states of the zero automaton. The implication $(3)\implies (1)$ follows from Proposition~\ref{Pro : ZeroAutoAcceptsReprOf0}.

Let us show that $(1)\implies (2)$. Suppose that the set $Z(\B,\boldsymbol{D})$ is accepted by a finite Büchi automaton $\mathcal{A}=(Q,q_0,F,\cup_{i=0}^{p-1}D_i,E)$. In view of Theorem~\ref{Thm : SetZeroSpectrumAccZeroAutomatonComplex}, it suffices to construct a finite Büchi $\mathcal{B}$ automaton accepting the set $Z(\pr,\Dig)$ in order to obtain that $X^{\Dig}(\pr)$ has no accumulation point in $\R$. Consider the finite Büchi automaton 
\[
	\mathcal{B}=(Q\times\{f,\overline{f}\}, (q_0,f_0),Q\times \{f\}, \Dig, E')
\]
where $f_0=f$ if the initial state $q_0$ is final and $f_0=\overline{f}$ otherwise, and the transitions in $E'$ are defined as follows. For $q,q'\in Q$, $x,x'\in \{f,\overline{f}\}$ and $a_0\in D_0,\ldots,a_{p-1}\in D_{p-1}$, there is a transition $((q,x),\sum_{i=0}^{p-1}a_{i}\beta_{i+1}\cdots\beta_{p-1},(q',x'))$ in $E'$ if there is a path labeled by $a_0\cdots a_{p-1}$ from $q$ to $q'$ in $\mathcal{A}$ and $x'=f$ if the path in $\mathcal{A}$ goes through a final state and  $x'=\overline{f}$ otherwise. 

We prove that $\mathcal{B}$ accepts $Z(\pr,\Dig)$. Consider $b\in Z(\pr ,\Dig)$. For all $n\in \N$, there exists $a_{n,0}\in D_0,\ldots,a_{n,p-1}\in D_{p-1}$ such that $b_n=\sum_{i=0}^{p-1}a_{n,i}\beta_{i+1}\cdots\beta_{p-1}$. Clearly, the infinite word $a=(a_{0,0}\cdots a_{0,p-1})(a_{1,0}\cdots a_{1,p-1})\cdots$ belongs to $Z(\B,\boldsymbol{D})$. Hence, there exists an accepting path labeled by $a$ in $\mathcal{A}$. Let $(q_n)_{n\in \N}$ be the sequence of states of this path. Then there is a path labeled by $b$ in $\mathcal{B}$ and going through the sequence of states $((q_{np},f_n))_{n\in \N}$ where for $n\in \N_{\ge1}$, $f_n=f$ if there exists $i\in [\![1,p]\!]$ such that $q_{(n-1)p+i}\in F$ and $f_n=\overline{f}$ otherwise. Since there are infinitely many $n$ such that $q_n\in F$, we obtain that there also are infinitely many $n$ such that $f_n=f$. Thus, the path in $\mathcal{B}$ labeled by $b$ going through the states $((q_{np},f_n))_{n\in \N}$ is accepting. 

Conversely, consider an infinite word $b$ over $\Dig$ accepted by $\mathcal{B}$. Let $((q_n,f_n))_{n\in \N}$ be the sequence of states of an accepting path labeled by $b$ in $\mathcal{B}$. By definition of the automaton $\mathcal{B}$, for all $n\in \N$, there exists $a_{n,0}\in D_0,\ldots,a_{n,p-1}\in D_{p-1}$ such that $b_n=\sum_{i=0}^{p-1}a_{n,i}\beta_{i+1}\cdots\beta_{p-1}$ and a path from $q_n$ to $q_{n+1}$ in $\mathcal{A}$ labeled by $a_{n,0}\cdots a_{n,p-1}$, and moreover, there is such path  going through a final state in $\mathcal{A}$ if and only if $f_n=f$. Hence, since there exist infinitely many $n$ such that $f_n=f$, there is an accepting path labeled by $a=(a_{0,0}\cdots a_{0,p-1})(a_{1,0}\cdots a_{1,p-1})\cdots$ in $\mathcal{A}$. Since $\mathcal{A}$ accepts the set $Z(\B,\boldsymbol{D})$, we get that $\sum_{n=0}^{+\infty} \frac{b_n}{\pr^{n+1}}=\val_{\B}(a)=0$.
\end{proof}

In the proof of Theorem~\ref{Thm : SetZeroSpectrumAccZeroAutomatonAlternate}, if the Büchi automaton $\mathcal{A}$ is deterministic, it is possible that the Büchi automaton $\mathcal{B}$ is not. This is not problematic since we do not require that the set $Z(\B,\boldsymbol{D})$ is accepted by a deterministic finite Büchi automaton. However, if the map $\boldsymbol{D}\to\Dig,\ (a_0,\ldots,a_{p-1}) \mapsto \sum_{i=0}^{p-1}a_{i}\beta_{i+1}\cdots\beta_{p-1}$ is injective and $\mathcal{A}$ is deterministic then $\mathcal{B}$ is deterministic as well.

\section{Periodicity of the expansions of $1$ in alternate bases and algebraicity of the bases}
\label{Sec : NecessaryConditionForPeriodicity}

The goal of this section is to prove the following result, which gives a necessary condition on $\B$ to be a Parry alternate base, that is, to have eventually periodic $\B^{(i)}$-expansions of $1$ for all $i\in\Int$. 

\begin{theorem}
\label{Thm : AlleventuallyPeriodicAlgebraicField}
If $\DBi{i}(1)$ is eventually periodic for all $i\in \Int$, then $\pr$ is an algebraic integer and $\beta_i \in \Q(\pr)$ for all $i\in \Int$.
\end{theorem}

In order to give intuition on the algebraic techniques that will be used in the proof, we start with an example.

\begin{example} 
\label{Ex : 3PhiPhi} 
Let $\boldsymbol{\beta}=(\beta_0,\beta_1,\beta_2)$ be a base such that the expansions of 1 are given by
\begin{equation}
\label{Eq : expansion3fifi}
  d_{\boldsymbol{\beta}}(1)=30^\omega,\quad
  d_{\boldsymbol{\beta^{(1)}}}(1)=110^\omega,\quad
  d_{\boldsymbol{\beta^{(2)}}}(1)=1(110)^\omega.
\end{equation}
We easily derive that $\beta_0,\beta_1,\beta_2$ satisfy the following set of equations
\[
	\frac{3}{\beta_0}=1,\quad 
	\frac{1}{\beta_1}+\frac{1}{\beta_1\beta_2}=1,\quad
	\frac{1}{\beta_2} + \left(\frac{1}{\beta_2\beta_0} + \frac{1}{\delta}\right)\frac{\delta}{\delta-1}=1,
\]
where $\delta=\beta_0\beta_1\beta_2$. Multiplying the first equation by $\delta$, the second one by $\beta_1\beta_2$ and the third one by $(\delta-1)\beta_2$, we obtain identities
\[
	3\beta_1\beta_2-\delta=0,\quad
	-\beta_1\beta_2+\beta_2+1=0,\quad
	\beta_1\beta_2+(2-\delta)\beta_2+\delta-1=0.
\]
In a matrix formalism, we have
\begin{equation}
\label{Eq : M3fifi}
\left(\begin{smallmatrix}
   3 	& 0 			& -\delta \\
   -1 	& 1 			& 1 \\
   1 	& 2-\delta 	& \delta-1
 \end{smallmatrix}\right)
 \left(\begin{smallmatrix}
   \beta_1\beta_2 \\
   \beta_2 \\
   1
 \end{smallmatrix}\right) = \left(\begin{smallmatrix}
   0 \\
   0 \\
   0
 \end{smallmatrix}\right).
\end{equation}
The existence of a non-zero vector $(\beta_1\beta_2,\beta_2,1)^T$ as a solution of this equation forces that the determinant of the coefficient matrix is zero, that is, that $ \delta^2-9\delta +9=0$. Hence we must have $\delta= \frac{9+3\sqrt5}{2}=3\varphi^2$ where $\varphi=\frac{1+\sqrt5}{2}$ is the golden ratio. Solving~\eqref{Eq : M3fifi} for this $\delta$, we obtain $\beta_1\beta_2=\frac{\delta}{3}=\varphi^2$ and $\beta_2=\beta_1\beta_2-1=\varphi^2-1=\varphi$, and finally $\beta_1=1+\frac{1}{\varphi}=\varphi$. Consequently, $\beta_0=\frac{\delta}{\beta_1\beta_2}=3$. Indeed, the triple ${\boldsymbol{\beta}}=(3,\varphi,\varphi)$ is an alternate base giving precisely~\eqref{Eq : expansion3fifi} as the expansions of 1, as already observed in~\cite{CharlierCisternino2021}.
\end{example}

For obtaining the values $\beta_0,\beta_1,\beta_2$ from the known eventually periodic expansions we have used the fact that $\beta_0,\beta_1,\beta_2$ and $\delta=\beta_0\beta_1\beta_2$ are solutions of a system of polynomial equations in four unknowns $x_0,x_1,x_2,y$, in our case
\[
	\left\{
	\begin{array}{rcl}
	3x_1x_2-y&=&0\\
	-x_1x_2+x_2+1&=&0\\
	x_1x_2 + (2-y)x_2+y-1&=&0\\
	x_1x_2x_3&=&y.
	\end{array}
	\right.
\]
The solution of the system yielded that $\delta$ is a root of a monic polynomial with integer coefficient, i.e., is an algebraic integer. The same strategy can be applied to any alternate basis where all the expansions $d_{\boldsymbol{\beta^{(i)}}}(1)$ are eventually periodic.

Recall that in case of classical R\'enyi expansion of $1$ with a single basis, it is easy to show that if $d_\beta(1)$ is eventually periodic (i.e., if $\beta$ is a Parry number) then $\beta$ is an algebraic integer. The monic polynomial with integer coefficient having $\beta$ as a root is called the Parry polynomial of $\beta$.

In the proof of Theorem~\ref{Thm : AlleventuallyPeriodicAlgebraicField}, we will work with formal power series whose coefficients are given by eventually periodic sequences. Let us prepare explicit form of these sums.

For given $m\in \N$, $k\in\N_{\ge1}$, we define $P_{m,k}$ as the set of polynomials in $\Z[y]$ of degree at most $m+k-1$ of the form 
\begin{align}
\label{eq:polynomialG}
   (y^k-1)(\sum_{n=0}^{m-1}a_ny^{m-1-n}) + \sum_{n=0}^{k-1}a_{m+n}y^{k-1-n}
\end{align}
where $a_0,\ldots,a_{m+k-1}\in\Z$. We say that the polynomial~\eqref{eq:polynomialG} is \emph{associated with the integers $a_0,\ldots,a_{m+k-1}$}. Note that this polynomial has maximal degree $m+k-1$ if $a_0\ne 0$.

\begin{lemma}
\label{Lem : PolyAssociatedWithUltPerWord}
Let $a$ be an eventually periodic sequence of integers with preperiod $m\in \N$ and period $k\in \N_{\ge1}$, that is, $a=a_0a_1\cdots a_{m-1}(a_ma_{m+1}\cdots a_{m+k-1})^\omega$. Then 
\[
  \sum_{n=0}^{+\infty}\frac{a_n}{y^{n+1}}
  =\frac{g}{y^m(y^k-1)}
\]
where $g$ is the polynomial in $P_{m,k}$ associated with the integers $a_0,\ldots,a_{m+k-1}$.
\end{lemma}

\begin{proof}
This is a straightforward verification.
\end{proof}

\begin{lemma}
\label{Lem : PolyAssociatedWithUltPerExpansion}
Suppose that $1$ has an eventually periodic $\B$-representation $a$ of preperiod $mp$ and period $kp$ with $m\in \N$ and $k\in\N_{\ge1}$. Then 
\[
	\pr^m(\pr^k-1) 
  	= \sum_{j=0}^{p-1}g_j(\pr)\beta_{j+1}\cdots\beta_{p-1}
\]
where for each $j\in\Int$, $g_j$ is the polynomial in $P_{m,k}$ associated with the integers $a_j,a_{j+p}\ldots,a_{j+(m+k-1)p}$. 
\end{lemma}

\begin{proof}
Rewrite~\eqref{Eq : valueAlternatBase} as
\[
  	1= \sum_{j=0}^{p-1} \sum_{n=0}^{+\infty} \frac{a_{np+j}}{\pr^{n+1}} \beta_{j+1}\cdots\beta_{p-1}.
\]
Since for every $j\in \Int$, the sequence $(a_{np+j})_{n\in\N}$ is eventually periodic with preperiod $m$ and period $k$, the result follows from Lemma~\ref{Lem : PolyAssociatedWithUltPerWord}.
\end{proof}

Whenever all $p$ expansions $\DBi{i}(1)$ are eventually periodic, for $i\in\Int$, we associate a system of polynomial equations, which we call the \emph{$\B$-polynomial system} by analogy to the \emph{$\beta$-polynomial} for real bases $\beta$~\cite{Parry1960}. We do this as follows. Without loss of generality, we suppose that for all $i\in \Int$, the expansion $\DBi{i}(1)$ has a preperiod $m_ip$ and a period $k_ip$ with $m_i\in\N$ and $k_i\in \N_0$. Then, for all $i\in\Int$, we let $g_{i,0},g_{i,1},\dots,g_{i,p-1}$ be the associated polynomials in $P_{m_i,k_i}$ as in Lemma~\ref{Lem : PolyAssociatedWithUltPerExpansion}, so that
\[
  	\pr^{m_i}(\pr^{k_i}-1) 
  	= \sum_{j=0}^{p-1}g_{i,j}(\pr)\beta_{i+j+1}\cdots\beta_{i+p-1}.
\]
For each $i\in\Int$, since the first digit of $\DBi{i}(1)$ is $\floor{\beta_i}\ge 1$, the degree of $g_{i,0}$ is $k_i+m_i-1$. The \emph{$\B$-polynomial system} is the system of $p+1$ polynomial equations in $p+1$ variables $x_0,x_1,\ldots,x_{p-1},y$ given by
\begin{equation}
\label{Eq : BPolynomials}
\begin{cases}
y^{m_i}(y^{k_i}-1)= 	\displaystyle{\sum_{j=0}^{p-1}}g_{i,j}x_{i+j+1}\cdots x_{i+p-1}, & \text{for } i\in\Int \\
y=\displaystyle{\prod_{i=0}^{p-1}}x_i
\end{cases}
\end{equation}
where, as usual, we use the convention $x_n=x_{n\bmod p}$ for $n\in \Z$. By construction, the $p$-tuple $(\beta_0,\ldots,\beta_{p-1},\pr)$ is a solution of this system.

\begin{example}
We resume Example~\ref{Ex : 3PhiPhi}. By writing each of the expansions from~\eqref{Eq : expansion3fifi} with a preperiod $3$ and a period $3$, that is, 
\[
	\DBi{0}(1)=300(000)^\omega, 
	\quad \DBi{1}(1)=110(000)^\omega 
	\quad \text{and}\quad \DBi{2}(1)=111(011)^\omega,
\]
we get $g_{0,0}=3(y-1)$, $g_{1,0}=g_{1,1}=g_{2,0}=y-1$, $g_{2,1}=g_{2,2}=y$ and $g_{0,1}=g_{0,2}=g_{1,2}=0$. The associated $\B$-polynomial system is
\[
	\begin{cases}
	y(y-1)=3(y-1)x_1x_2\\
	y(y-1)=(y-1)x_2x_0+(y-1)x_0\\
	y(y-1)=(y-1)x_0x_1 + y x_1+y\\
	y=x_0x_1x_2.
	\end{cases}
\]
By multiplying the second equation by $x_1x_2$ and the third one by $x_2$ and by substituting $x_0x_1x_2$ by $y$, we get the three equations
\[
	\begin{cases}
	y(y-1)=3(y-1)x_1x_2\\
	y(y-1)x_1x_2=y(y-1)x_2+y(y-1)\\
	y(y-1)x_2=(y-1)y + y x_1x_2+y x_2.
	\end{cases}
\]
Placing the first equation in the last line, this can be rewritten as
\[
	\begin{pmatrix}
	-y(y-1)	& y(y-1) 	& y(y-1) \\
	y		& y-y(y-1)	& y(y-1) \\
	3(y-1)	& 0 			& -y(y-1)
	\end{pmatrix}
	\begin{pmatrix}
	x_1x_2\\
	x_2\\
	1
	\end{pmatrix}=
	\begin{pmatrix}
	0\\0\\0
	\end{pmatrix}.
\]
The matrix of this system is equal to
$M(y)-y(y-1) I_3$ where
\[
	M(y)
	=\begin{pmatrix}
	g_{1,2}	& yg_{1,0}	& yg_{1,1}\\
	g_{2,1}	& g_{2,2}	& yg_{2,0}\\
	g_{0,0}	& g_{0,1} 	& g_{0,2}
	\end{pmatrix}
\]
and $I_3$ is the identity matrix of size $3$.
\end{example}

\begin{proof}[Proof of Theorem~\ref{Thm : AlleventuallyPeriodicAlgebraicField}]
Let $m\in \N$ and $k\in \N_0$ be such that the expansions $\DBi{i}(1)$ all have preperiod $mp$ and period $kp$, for $i\in \Int$. Then we consider the associated polynomial system~\eqref{Eq : BPolynomials}, where $m_i=m$ and $k_i=k$ for all $i\in\Int$. We index the equations of this system from $0$ to $p$. For each $i\in [\![1,p-1]\!]$, we multiply the $i$th equation by $\prod_{k=i}^{p-1}x_k$, which becomes
\[
	y^{m}(y^k-1)\prod_{k=i}^{p-1}x_k
	= \sum_{j=0}^{p-1}(g_{i,j}\prod_{k=i+j+1}^{2p-1}x_k).
\]
By substituting $x_0\cdots x_{p-1}$ by $y$, the latter equation can be rewritten as
\[
	y^{m}(y^k-1)\prod_{k=i}^{p-1}x_k
	= \sum_{j=0}^{p-i-1}(y g_{i,j} \prod_{k=i+j+1}^{p-1}x_k)
	+\sum_{j=p-i}^{p-1}(g_{i,j} \prod_{k=i+j+1-p}^{p-1}x_k).
\]
Now, the first $p$ equations of the system can be written in the matrix form
\begin{equation}
\label{Eq : SystemForAlgebraic}
(M(y)-y^m(y^k-1)I_p) \overrightarrow{v}(x_1,\ldots,x_{p-1})=\vec{0}
\end{equation}
where $I_p$ is the identity matrix of size $p$, $\vec{0}$ is the zero column vector of size $p$, 
\[
	\overrightarrow{v}(x_1,\ldots,x_{p-1})
	=\begin{pmatrix}
	x_1x_2\cdots x_{p-1}\\ 
	x_2\cdots x_{p-1}\\ 
	\vdots\\
	x_{p-1}\\
	1
	\end{pmatrix}
\]
and 
\[
	M(y)=
	\begin{pmatrix}
	g_{1,p-1} & yg_{1,0}	  & \cdots &	yg_{1,p-3} & yg_{1,p-2} \\
	g_{2,p-2} & g_{2,p-1} & \cdots & yg_{2,p-4} & yg_{2,p-3} \\
	\vdots   & \vdots 	& \ddots & \vdots & \vdots \\
	g_{p-1,1} 	& g_{p-1,2}   	& \cdots & g_{p-1,p-1}   & yg_{p-1,0} \\
	g_{0,0}   	& g_{0,1}     	& \cdots & g_{0,p-2}     & g_{0,p-1}
	\end{pmatrix}
\]
Since $(\beta_0,\ldots,\beta_{p-1},\delta)$ is a non trivial solution of the original system, we get that $\delta$ is a root of the polynomial $h=\det(M(y)-y^m(y^k-1)I_p)$ of $\Z[y]$. By construction, for every $i,j\in\Int$, the polynomial $g_{i,j}$ has degree at most $m+k-1$. Therefore, the highest degree of $h$ is obtained from the product $\prod_{i=0}^{p-1}(g_{i,p-1}{-}y^{m}(y^k {-}1))$. This shows that $h$ has leading coefficient $(-1)^p$. Since $\pr$ is a root of $h$, we get that $\pr$ is an algebraic integer. 

It remains to prove that $\beta_i\in \Q(\pr)$ for all $i\in \Int$. To that purpose, we will apply the famous Perron-Frobenius theorem. First, thanks to Lemma~\ref{Lem : PolyAssociatedWithUltPerWord}, we know that the matrix $M(\pr)$ has non-negative entries. Then, by Lemma~\ref{Lem : PolyAssociatedWithUltPerExpansion} and since any $\B$-expansion starts with a non-zero digit, the entries $\pr g_{1,0}(\pr)$, $\pr g_{2,0}(\pr),\ldots,\pr g_{p-1,0}(\pr)$, $g_{0,0}(\pr)$ of $M(\pr)$ in respective positions $(0,1)$, $(1,2),\ldots,(p-2,p-1)$, $(p-1,0)$ are positive. Therefore, the matrix $M(\pr)$ is irreducible. By the Perron-Frobenius theorem, the vector $\overrightarrow{v}(\beta_1,\ldots,\beta_{p-1})$ is the unique positive eigenvector of $M(\pr)$ having $1$ as its last entry and the corresponding eigenvalue $\pr^m(\pr^k-1)$ is the Perron-Frobenius eigenvalue of $M(\pr)$. Moreover, the rank of the matrix $M(\pr)-\pr^m(\pr^k-1)I$ is $p-1$. Thus, the corresponding linear system in the unknowns $c_1=x_1x_2\cdots x_{p-1}$, $c_2=x_2\cdots x_{p-1},\ldots,c_{p-1}=x_{p-1}$ is equivalent to that obtained by deleting one its $p$ equations. The obtained system has full rank $p-1$. Since all entries of $M(\pr)-\pr^m(\pr^k-1)I$ belong to the field $\Q(\pr)$, any solution vector of the latter system has components $c_i$ in $\Q(\pr)$. Hence, the products $\beta_1\beta_2\cdots \beta_{p-1}, \beta_2\cdots \beta_{p-1}, \ldots, \beta_{p-1}$ all belong to $\Q(\pr)$. We obtain in turn that $\beta_1,\ldots,\beta_{p-1}\in\Q(\pr)$. Since moreover $\beta_0=\pr/(\beta_1\cdots\beta_{p-1})$, we also get that $\beta_0\in\Q(\pr)$.
\end{proof}

Let us emphasize that the greediness of the representations was not necessary in the proof of Theorem~\ref{Thm : AlleventuallyPeriodicAlgebraicField}. We only need that each $\B^{(i)}$-representation of $1$ starts with a non-zero digit. Therefore, we have actually proved the following stronger result.

\begin{theorem}
If $1$ has eventually periodic $\B^{(i)}$-representations for all $i\in\Int$, then $\pr$ is an algebraic integer. If moreover these $p$ representations have non-negative digits and they all start with a non-zero digit, then $\beta_0,\ldots,\beta_{p-1} \in \Q(\pr)$.
\end{theorem}

From the proof of Theorem~\ref{Thm : AlleventuallyPeriodicAlgebraicField}, we deduce the following result about the uniqueness of the base. 

\begin{proposition}
Suppose that $\boldsymbol{\alpha}=(\alpha_0,\ldots,\alpha_{p-1})$ and $\B=(\beta_0,\ldots,\beta_{p-1})$ are two alternate bases such that $\prod_{i=0}^{p-1}\alpha_i=\prod_{i=0}^{p-1}\beta_i$, and suppose that there exists $p$ eventually periodic sequences $a^{(0)},\ldots,a^{(p-1)}$ of non-negative integers such that $a_0^{(i)}\ge1$ and $\val_{\boldsymbol{\alpha}^{(i)}}(a^{(i)})=\val_{\boldsymbol{\beta}^{(i)}}(a^{(i)})=1$ for every $i\in\Int$. Then $\boldsymbol{\alpha}=\B$.
\end{proposition}

\begin{proof}
Using the same notation as in the proof of Theorem~\ref{Thm : AlleventuallyPeriodicAlgebraicField}, given the product $\pr=\prod_{i=0}^{p-1}\beta_i$, the vector $\overrightarrow{v}(\beta_1,\ldots,\beta_{p-1})$ is the unique positive eigenvector of $M(\pr)$ having $1$ as its last entry. Therefore, we must have $\overrightarrow{v}(\alpha_1,\ldots,\alpha_{p-1})=\overrightarrow{v}(\beta_1,\ldots,\beta_{p-1})$, hence $\alpha_i=\beta_i$ for all $i\in[\![1,p]\!]$. Moreover, we have $\alpha_0=\pr/(\alpha_1\cdots\alpha_{p-1})=\pr/(\beta_1\cdots\beta_{p-1})=\beta_0$.
\end{proof}

In particular, we get the following two corollaries.

\begin{corollary}
\label{Cor : UniquenessDB}
Let $\boldsymbol{\alpha}=(\alpha_0,\ldots,\alpha_{p-1})$ and $\B=(\beta_0,\ldots,\beta_{p-1})$ be two alternate bases such that $\prod_{i=0}^{p-1}\alpha_i=\prod_{i=0}^{p-1}\beta_i$ and suppose that for every $i\in\Int$, the $\boldsymbol{\alpha}^{(i)}$-expansion of $1$ and $\B^{(i)}$-expansions of $1$ coincide and are eventually periodic. Then $\boldsymbol{\alpha}=\B$.
\end{corollary}

\begin{corollary}
If $\DBi{i}(1)=\DB(1)$ for all $i\in \Int$ and $\DB(1)$ is eventually periodic, then $\beta_i=\beta_0$ for all $i\in\Int$.
\end{corollary}

\begin{proof}
Apply Corollary~\ref{Cor : UniquenessDB} to $\B$ and $\B^{(1)}$.
\end{proof}

\section{Spectrum and periodicity of the expansions of $1$ in alternate bases}
\label{Sec : SpectrumPeriodicity}

In Section~\ref{Sec : NecessaryConditionForPeriodicity}, we have derived a necessary condition for the $\B^{(i)}$-expansions to be all eventually periodic, i.e., for $\B$ to be a Parry alternate base. Namely that the product $\pr$ of the bases is an algebraic integer and all $\beta_j$ belong to the field $\Q(\pr)$. In this section, we give a sufficient condition. 

We adopt the same notation and convention as in Section~\ref{Sec : SpectrumAlternate}: we fix an alternate base $\B=(\beta_0,\ldots,\beta_{p-1})$, we set $\pr=\prod_{i=0}^{p-1}\beta_i$, we consider a fixed $p$-tuple $\boldsymbol{D}=(D_0,\ldots,D_{p-1})$ where every $D_i$ is an alphabet of integers containing $0$ and we let $\Dig$ be the corresponding alphabet or real numbers as defined in~\eqref{Eq : Digits}.

\begin{proposition}
\label{Pro : NoAccPointAllUltPer}
If $D_i\supseteq[\![-\floor{\beta_i},\floor{\beta_i}]\!]$ for all $i\in \Int$ and if the spectrum $X^{\Dig}(\pr)$ has no accumulation point in $\R$, then $\DBi{i}(1)$ is eventually periodic for all $i\in\Int$. 
\end{proposition}

\begin{proof}
Suppose that $\DB(1)$ is not eventually periodic. Then the sequence of remainders $(r_{\ell p-1}(1))_{\ell\in \N}$ is injective. For all $x\in [0,1]$ and $\ell\in \N$, we have 
\begin{equation}
\label{Eq : Remainders}
	r_{\ell p-1}(x)=\pr^\ell x-\sum_{n=0}^{\ell-1}d_n\pr^{\ell-1-n}
\end{equation}
where $d_n=\sum_{i=0}^{p-1}\varepsilon_{np+i}(x)\beta_{i+1}\cdots\beta_{p-1}$. Since $D_i\supseteq[\![-\floor{\beta_i},\floor{\beta_i}]\!]$ for each $i\in \Int$, we get that for all $\ell\in\N$, the remainder $r_{\ell p-1}(1)$ is an element of $X^{\Dig}(\pr)$. 
Since the remainders all belong to the interval $[0,1)$, the spectrum $X^{\Dig}(\pr)$ has an accumulation point in $\R$. By Lemma~\ref{Lem : SpectrumAccAllSpectraAcc}, either all the spectra $X(i)$ based on the cyclic shifts $\B^{(i)}$ of the alternate base and the corresponding shifted $p$-tuple of alphabets $\boldsymbol{D}^{(i)}$ for $i\in\Int$ have an accumulation point or none of them has. The result follows.
\end{proof}

\begin{proposition}
\label{Pro : PisotExtendedFieldThenNoAccPoint}
If $\pr$ is a Pisot number and $\beta_i\in \Q(\pr)$ for all $i\in \Int$ then the spectrum $X^{\Dig}(\pr)$ has no accumulation point in $\R$.
\end{proposition}

\begin{proof}
Since the set $\Dig$ is a finite subset of $\Q(\pr)$ and $\pr$ is an algebraic integer, there exist a positive integer $q$ and a finite subset $A$ of the ring of algebraic integers in $\Q(\pr)$ such that 
$\Dig\cup(\Dig-\Dig)=\frac{1}{q}A$.
Let $x,y\in X^{\Dig}(\pr)$ such that $x\neq y$. There exists $\ell\in \N$ and $a_0,\ldots,a_{\ell-1}\in A$ such that
\[
	x-y = \frac{1}{q}\sum_{n=0}^{\ell-1} a_n\pr^n.
\]
We obtain that $q(x-y)$ is an algebraic integer. 
Let $d$ denote the (algebraic) degree of $\pr$ and let $\pr_2,\ldots,\pr_d$ be the Galois conjugates of $\pr$. Moreover, set $\pr_1=\pr$. Then
\[
	1\leq \big| N_{\Q(\pr)/\Q}(q(x-y))\big| 
	= q|x-y| \prod_{k=2}^d |\psi_k(q(x-y))| .
\]
Since $\pr$ is a Pisot number, for all $k\in [\![2,d]\!]$, we have $|\pr_k|<1$ and hence
\[
	|\psi_k(q(x-y))| 
	\le M\sum_{n=0}^{\ell-1} |\pr_k|^n
	\le  \frac{M}{1-|\pr_k|}
\]
where $M=\max\{|\psi_k(a)| : k\in [\![2,d]\!],\, a\in A\}$. We get that
\[
	|x-y| \ge \frac{1}{q}\prod_{k=2}^{d} \frac{1-|\pr_k|}{M}.
\]
The latter inequality states that the distance between distinct elements $x,y$ of the spectrum $X^{\Dig}(\pr)$ is bounded from below by a constant uniformly for all pairs $x,y$.
\end{proof}

As a consequence, we get the following theorem, which for the case $p=1$ is a well-known result of Schmidt~\cite{Schmidt1980}. 

\begin{theorem}
\label{Thm : PisotExtendedFieldThenUltPer}
If $\pr$ is a Pisot number and $\beta_i\in \Q(\pr)$ for all $i\in \Int$ then $\DBi{i}(1)$ is eventually periodic for all $i\in \Int$.
\end{theorem}

\begin{proof}
First apply Proposition~\ref{Pro : PisotExtendedFieldThenNoAccPoint} with $\boldsymbol{D}=([\![-\floor{\beta_0},\floor{\beta_0}]\!],\ldots,[\![-\floor{\beta_{p-1}},\floor{\beta_{p-1}}]\!])$ and then apply Proposition~\ref{Pro : NoAccPointAllUltPer}.
\end{proof}

Let us make several remarks concerning the previous result. First, the condition of $\pr$ being a Pisot number is neither sufficient nor necessary for $\B$ to be a Parry alternate base, i.e., in order to have that $\DBi{i}(1)$ is eventually periodic for all $i\in \Int$. Indeed, it is not necessary even for $p=1$ since there exist Parry numbers which are not Pisot. To see that it is not sufficient for $p\ge 2$, consider the alternate base $\B=(\sqrt{\beta},\sqrt{\beta})$ where $\beta$ is the smallest Pisot number. The product $\pr$ is the Pisot number $\beta$. However, the $\B$-expansion of $1$ is equal to $d_{\sqrt{\beta}}(1)$, which is known to be aperiodic. This follows from the fact that the only Galois conjugate of $\sqrt{\beta}$ is $-\sqrt{\beta}$, and thus $\sqrt{\beta}$ is not a Perron number, hence not a Parry number either.

Furthermore, the bases $\beta_0,\ldots,\beta_{p-1}$ need not be algebraic integers in order to have the property that $\DBi{i}(1)$ is eventually periodic for all $i\in \Int$. To see this, consider for instance the alternate base $\B=(\frac{1+\sqrt{13}}{2},\frac{5+\sqrt{13}}{6})$ from Example~\ref{Ex : BaseHabituelle}. For this base, we have $\DBi{0}(1)=2010^\omega$ and $\DBi{1}(1)=110^\omega$. However, $\frac{5+\sqrt{13}}{6}$ is not an algebraic integer.

As illustrated in the following example, for the same non Pisot algebraic integer $\pr$, there may exist two length-$p$ alternate bases $\boldsymbol{\alpha}=(\alpha_0,\cdots,\alpha_{p-1})$ and $\B=(\beta_0\cdots\beta_{p-1})$ such that $\alpha_0\cdots\alpha_{p-1}=\beta_0\cdots\beta_{p-1}=\pr$ and for all $i\in \Int$, the expansion $d_{\boldsymbol{\alpha}^{(i)}}(1)$ is eventually periodic whereas there exists $i\in \Int$ such that $\DBi{i}(1)$ is not. The technique used for showing aperiodicity is inspired by the work~\cite{LiaoSteiner2012}.

\begin{example}
Consider the real root $\pr>1$ of the polynomial $x^6-x^5-1$. This number is an algebraic integer but it is not a Pisot number since two of its Galois conjugates have modulus greater than $1$. Consider the alternate base $\boldsymbol{\alpha}=(\frac{1+\pr^7}{\pr^7},\frac{\pr^8}{1+\pr^7})$. We can compute that $d_{\boldsymbol{\alpha}}(1)=10^{13}10^\omega$ and $d_{\boldsymbol{\alpha}^{(1)}}(1)=10^{18}10^{20}(10^{27})^\omega$. Now consider $\B=(\frac65,\frac56\pr)$. We prove that $\DB(1)$ is not eventually periodic. Let $\gamma$ be a Galois conjugate of $\pr$ such that $|\gamma|>1$ and let $\psi \colon \Q(\pr) \to \Q(\gamma)$ be the corresponding field isomorphism induced by $\psi(\delta)=\gamma$. We prove that $(r_{12n-1}(1))_{n\in\N}$ is not eventually periodic, where we set $r_{-1}(1)=1$. To do so, it is enough to prove that $\big(|\psi(r_{12n-1}(1))|\big)_{n\in\N}$ is eventually strictly increasing. It can be computed that the word $10^{12}$ is a prefix of $\DB^*(1)$ and $\DBi{1}^*(1)$. Therefore, by Theorem~\ref{thm:Parry} and Equality~\eqref{Eq : Remainders}, for all $x\in [0,1]$, we get 
\[
	r_{11}(x)\in \{\pr^6x\}\cup
			\{\pr^6x-\beta_1\pr^k : k\in[\![0,5]\!]\}
			\cup \{\pr^6x-\pr^k : k\in[\![0,5]\!]\}.
\]
Hence, for all $x\in [0,1]\cap \Q(\pr)$, we have 
\[
	\psi(r_{11}(x))\in \{\gamma^6\psi(x)\}\cup
			\{\gamma^6\psi(x)-\frac56\gamma^{k+1} : k\in[\![0,5]\!]\}
			\cup\{\gamma^6\psi(x)-\gamma^{k} : k\in[\![0,5]\!]\}.
\]
Since $|\gamma|\le \frac{6}{5}$, we get
\[
	|\psi(r_{11}(x))|\ge |\gamma|^6|\psi(x)|-|\gamma|^5.
\]
Thus, if we have 
\[
	|\psi(x)|>\frac{|\gamma|^5}{|\gamma|^6-1} \simeq 5.49
\]
then we obtain $|\psi(r_{11}(x))|>|\psi(x)|$. 
It can be computed that $10^{13}10^{15}10^{13}10^{27}10^{11}$ is the prefix of $\DB(1)$ of length $84$. Hence, by using~\eqref{Eq : Remainders} again, we get
\[r_{83}(1)=\pr^{42}-\beta_1\pr^{41}-\beta_1\pr^{34}-\beta_1\pr^{26}-\beta_1\pr^{19}-\beta_1\pr^5.\]
This implies 
\[
	\psi(r_{83}(1))
	=\gamma^{42}-\frac56\gamma^{42}-\frac56\gamma^{35}
	-\frac56\gamma^{27}-\frac56\gamma^{20}-\frac56\gamma^6.
\]
Now for $x=r_{83}(1)$, we have $|\psi(x)|\simeq 6.23 >5.49$. We get $|\psi(r_{11}(x))|>|\psi(x)|$ where $r_{11}(x)=r_{95}(1)$. Iterating the argument, we obtain that the sequence $(|\psi(r_{12n-1}(1))|)_{n\ge 7}$ is strictly increasing. 
\end{example}

\section{Alternate bases whose set of zero representations is accepted by a finite Büchi automaton}

Once again, we use the notation introduced in Section~\ref{Sec : SpectrumAlternate}: namely we use fixed $\B,\pr,\boldsymbol{D}$ and then we work with the corresponding digit set $\Dig$, set of representations of zero $Z(\B,\boldsymbol{D})$ and spectrum $X^\Dig(\pr)$. We combine the previously established results in order to characterize for which alternate bases the set $Z(\B,\boldsymbol{D})$ is accepted by a finite Büchi automaton. In doing so, we generalize Theorem~\ref{Thm : EquivalencesZeroReprFrougnyPelantova} to alternate bases. We need one more lemma.

\begin{lemma}
\label{Lem : ProductNotPisotSpectrumAcc}
If the spectrum $X^{\Dig}(\pr)$ has no accumulation point in $\R$ and if there exists $j\in \Int$ such that $[\![-\ceil{\pr}+1,\ceil{\pr}-1]\!]\subseteq D_j$, then $\pr$ is a Pisot number.
\end{lemma}

\begin{proof}
Suppose that $X^{\Dig}(\pr)$ has no accumulation point in $\R$ and let $j$ be an index as in the statement. Since $X^{\Dig}(\pr)=\sum_{i=0}^{p-1} X^{D_i}(\pr)\beta_{i+1}\cdots \beta_{p-1}$, the spectrum $X^{D_j}(\pr)$ has no accumulation point in $\R$. By hypothesis on $j$, the spectrum $X^{\ceil{\pr}-1}(\pr)$ has no accumulation point in $\R$ either. By Theorem~\ref{Thm : AccPointAkiyamaKomornikFeng}, we get that $\pr$ is a Pisot number.
\end{proof}

\begin{theorem}
\label{Thm : MainEquivalences}
The following assertions are equivalent.
\begin{enumerate}
\item The set $Z(\B,\boldsymbol{D})$ is accepted by a finite Büchi automaton for all $p$-tuple of alphabets of integers $\boldsymbol{D}=(D_0,\ldots,D_{p-1})$.
\item The set $Z(\B,\boldsymbol{D})$ is accepted by a finite Büchi automaton for one $p$-tuple of alphabets of integers $\boldsymbol{D}=(D_0,\ldots,D_{p-1})$ such that $D_i\supseteq [\![-\floor{\beta_i},\floor{\beta_i}]\!]$ for all $i\in \Int$ and $\floor{\beta_j}\ge\ceil{\pr}-1$ for some $j\in \Int$.
\item $\pr$ is a Pisot number and $\beta_i\in \Q(\pr)$ for all $i\in \Int$.
\end{enumerate}
\end{theorem}

\begin{proof}
The implication $(1)\implies (2)$ is straightforward. Now, suppose that $(2)$ holds. By Theorem~\ref{Thm : SetZeroSpectrumAccZeroAutomatonAlternate} and Proposition~\ref{Pro : NoAccPointAllUltPer}, the greedy expansions $\DBi{i}(1)$ are eventually periodic for all $i\in \Int$. Then, by Theorem~\ref{Thm : AlleventuallyPeriodicAlgebraicField}, we get that $\pr$ is an algebraic integer and $\beta_i\in \Q(\pr)$ for all $i\in \Int$. Moreover, since there exists $j\in \Int$ such that $\floor{\beta_j}\ge \ceil{\pr}-1$, we obtain from Theorem~\ref{Thm : SetZeroSpectrumAccZeroAutomatonAlternate} and  Lemma~\ref{Lem : ProductNotPisotSpectrumAcc} that $\pr$ is a Pisot number. Hence, we have shown that $(2)\implies (3)$. Finally, the implication $(3)\implies (1)$ is obtained by combining Proposition~\ref{Pro : PisotExtendedFieldThenNoAccPoint} and Theorem~\ref{Thm : SetZeroSpectrumAccZeroAutomatonAlternate}.
\end{proof}

\section{Normalization in alternate base}
\label{Sec : Normalization}

In this section, we apply our results in order to show that the normalization in alternate base is computable by a finite Büchi automaton under certain hypotheses, in which case we construct such an automaton. The \emph{normalization function} $\nu_{\B,\boldsymbol{D}}\colon (\cup_{i=0}^{p-1}D_i)^\N \to (\cup_{i=0}^{p-1}[\![0,\ceil{\beta_i}-1]\!])^\N$ is the partial function mapping any $\B$-representation $a\in\otimes_{n\in\N}D_n$ of a real number $x\in[0,1)$ to the $\B$-expansion of $x$. We say that $\nu_{\B,\boldsymbol{D}}$ is \emph{computable by a finite Büchi automaton} if there exists a finite Büchi automaton accepting
the set
\[
	\{ (u,v)\in \bigotimes_{n\in\N}(D_n\times [\![0,\ceil{\beta_n}-1]\!] ) 
: \val_{\B}(u)=\val_{\B}(v) \text{ and }   \exists x\in [0,1),\ v=\DB(x)\}.
\]

Following the same lines as in the real base case, we start by constructing a \emph{converter} by using the zero automaton $\mathcal{Z}(\B,\boldsymbol{D})=(Q_{\B,\boldsymbol{D}}, (0,0),Q_{\B,\boldsymbol{D}}, \cup_{i=0}^{p-1}D_i,E)$ defined in Section~\ref{Sec : SpectrumAlternate}. Consider two $p$-tuples of alphabets $\boldsymbol{D}=(D_0,\ldots,D_{p-1})$ and $\boldsymbol{D'}=(D'_0,\ldots,D'_{p-1})$. We denote the $p$-tuple of alphabets $(D_0- D_0',\ldots,D_{p-1}- D'_{p-1})$ by $\boldsymbol{D}-\boldsymbol{D'}$. The \emph{converter from $\boldsymbol{D}$ to $\boldsymbol{D}'$} is the Büchi automaton
\[
	\mathcal{C}(\B,\boldsymbol{D},\boldsymbol{D'})
	=(Q_{\B,\boldsymbol{D}-\boldsymbol{D}'},(0,0),Q_{\B,\boldsymbol{D}-\boldsymbol{D}'},\cup_{i=0}^{p-1}(D_i\times D_i'),E')
\]
where $E'$ is the set of transitions defined as follows: for $(i,s),(j,t)\in Q_{\B,\boldsymbol{D}-\boldsymbol{D'}}$ and for $\couple{a}{b}\in \cup_{i=0}^{p-1}(D_i\times D_i')$, there is a transition 
\[
	(i,s) \xrightarrow[ \quad ]{\couple{a}{b}} (j,t)
\] 
if and only if $\couple{a}{b}\in D_i\times D_i'$ and there is a transition $(i,s) \xrightarrow[ \quad ]{a-b} (j,t)$ in $\mathcal{Z}(\B,\boldsymbol{D}-\boldsymbol{D}')$.

\begin{proposition}
The converter $\mathcal{C}(\B,\boldsymbol{D},\boldsymbol{D}')$ accepts the set 
\[
	\{\couple{u}{v}\in \bigotimes_{n\in\N}(D_n\times D_n') : 
	\val_{\B}(u)=\val_{\B}(v)\}.
\]
\end{proposition}

\begin{proof}
This is a direct consequence of Proposition~\ref{Pro : ZeroAutoAcceptsReprOf0}.
\end{proof}

\begin{proposition}
\label{Pro : ConverterFinite}
If $\pr$ is a Pisot number and $\beta_i\in \Q(\pr)$ for all $i\in \Int$, then the converter $\mathcal{C}_{\B,\boldsymbol{D},\boldsymbol{D'}}$ is finite.
\end{proposition}

\begin{proof}
By Theorems~\ref{Thm : MainEquivalences} and~\ref{Thm : SetZeroSpectrumAccZeroAutomatonAlternate}, the zero automaton $\mathcal{Z}(\B,\boldsymbol{D}-\boldsymbol{D'})$ is finite. Hence, so is the converter $\mathcal{C}_{\B,\boldsymbol{D},\boldsymbol{D'}}$. 
\end{proof}

In the case where $\B$ is a Parry alternate base, i.e., where all $\qDBi{i}(1)$ are eventually periodic, a deterministic finite automaton $\A_{\B}$ accepting $\Fac(\{\DB(x) : x\in[0,1)\})$ was built in \cite{CharlierCisternino2021}, thus showing that the corresponding $\B$-shift is sofic. Here, we consider a modification of this automaton in order to get a Büchi  automaton accepting $\{\DB(x) : x\in[0,1)\}$. Suppose that 
\[
	\qDBi{i}(1)=t_0^{(i)}\cdots t_{m_i-1}^{(i)} \big(t_{m_i}^{(i)}\cdots t_{m_i+n_i-1}^{(i)}\big)^\omega
\]
for all $i\in \Int$. Without loss of generality, we suppose that $\DBi{i}(1)$ has a non-zero preperiod for all $i\in \Int$, i.e., in the case of a purely periodic expansion $(t_0\cdots t_{n-1})^\omega$, we work with the writing $t_0(t_1\cdots t_{n-1}t_0)^\omega$ instead. Then we define the Büchi automaton
\[
	\mathcal{B}_{\B}=(Q,(0,0,0),F,[\![0,\max_{0\le i<p}\ceil{\beta_i}-1]\!],E)
\] 
where the set of states is 
\[
	Q=\{(i,j,k) : 
		i,j\in\Int,\ k\in[\![0,m_i+n_i-1]\!]\},
\]
the set of final states is
\[
	F=\big\{(i,i,0) : i\in\Int\big\}
\] 
and the (partial) transition function $E\colon Q\times [\![0,\max_{0\le i<p}\ceil{\beta_i}-1]\!]\to Q$ is defined as follows: for each $i,j\in\Int$ and each $k\in[\![0,m_i+n_i-1]\!]$, 
we have 
\[
	E((i,j,k),t_k^{(i)})=
	\begin{cases}
		(i,(j+1)\bmod p,k+1), & \text{if }k\ne m_i+n_i-1;\\
		(i,(j+1)\bmod p,m_i), & \text{otherwise}
	\end{cases}
\]
and for all $s\in[\![0,t_k^{(i)}-1]\!]$, we have
\[
	E((i,j,k),s)=((j+1)\bmod p,(j+1)\bmod p,0).
\]

\begin{proposition}
\label{Pro : BuchiDB}
If $\DBi{i}(1)$ is eventually periodic for all $i\in \Int$ then the Büchi automaton $\mathcal{B}_{\B}$ accepts the set $\{d_{\B}(x) : x\in[0,1)\}$.
\end{proposition}

\begin{proof}
For an alternate base $\B$ of length $p$ with $\qDB(1)=t_0t_1t_2\cdots$ and for $j\in\Int$, we define the set
\[
	Y(\B,j)=\{t_0\cdots t_{\ell-2}s :
	\ell\in\N_{\ge 1},\
	\ell\bmod p= j,\
	t_{\ell-1}>0,\
	s\in[\![0,t_{\ell-1}-1]\!]
	\}.
\]
An infinite word is accepted by $\mathcal{B}_{\B}$ if and only if it can be factored as $u_0u_1u_2\cdots$ where each factor $u_n$ corresponds to a first return to a final state, i.e., for all $n\in\N$, there is a path labeled by $u_n$ from a state of the form $(i,i,0)$ to a state of the form $(j,j,0)$ and $u_n$ is the shortest next factor with this property. Since we have built $\mathcal{B}_{\B}$ by using a non-zero preperiod for each $\qDBi{i}(1)$, each such factor $u_n$ must belong to the set $ Y(\B^{(|u_0|+\cdots+|u_{n-1}|)},|u_n|)$. Proposition~45 in \cite{CharlierCisternino2021} states that
\[
	\{\DB(x) : x\in[0,1)\}
	=\bigcup_{h_0=0}^{p-1}Y(\B,h_0)
	\Bigg( \bigcup_{h_1=0}^{p-1}Y(\B^{(h_0)},h_1)
	\Bigg( \bigcup_{h_2=0}^{p-1}Y(\B^{(h_0+h_1)},h_2)
	\Bigg(
	\cdots
	\Bigg)\Bigg)\Bigg).
\]
The conclusion follows.
\end{proof}

\begin{theorem}
\label{Thm : Normalization}
If $\pr$ is a Pisot number and $\beta_i\in \Q(\pr)$ for all $i\in \Int$, then the normalization function $\nu_{\B,\boldsymbol{D}}$ is computable by a finite Büchi automaton.
\end{theorem}

\begin{proof}
If $\pr$ is a Pisot number and $\beta_i\in \Q(\pr)$ for all $i\in \Int$, then by Theorem~\ref{Thm : PisotExtendedFieldThenUltPer}, the greedy $\B^{(i)}$-expansions of $1$ are eventually periodic for all $i\in \Int$.  By Proposition~\ref{Pro : BuchiDB}, the finite Büchi automaton $\mathcal{B}_{\B}$ accepts the set $\{\DB(x) : x\in[0,1)\}$. Thanks to this automaton, we construct a finite Büchi automaton accepting the set 
\[
	\big\{ (u,v)\in \bigotimes_{n=0}^{+\infty}(D_n\times [\![0,\ceil{\beta_n}-1]\!] ) : \exists x\in [0,1),\ v=\DB(x)\big \}.
\]
By intersecting the latter Büchi automaton with the converter $\mathcal{C}(\B,\boldsymbol{D},\boldsymbol{D}')$ where $\boldsymbol{D}'=([\![0,\ceil{\beta_0}-1]\!],\ldots,[\![0,\ceil{\beta_{p-1}}-1]\!])$, which is finite by Proposition~\ref{Pro : ConverterFinite}, we get a finite Büchi automaton accepting the set
\[
	\{ (u,v)\in \bigotimes_{n=0}^{+\infty}(D_n\times [\![0,\ceil{\beta_n}-1]\!] )
: \val_{\B}(u)=\val_{\B}(v) \text{ and }   \exists x\in [0,1),\ v=\DB(x)\}.  
\]
\end{proof}

\begin{example}
Consider again the alternate base $\B=(\frac{1+ \sqrt{13}}{2},\frac{5+ \sqrt{13}}{6})$. We have $\DBi{0}(1)=2010^\omega$ and $\DBi{1}(1)=110^\omega$, hence $\qDBi{0}(1)=200(10)^\omega$ and $\qDBi{1}(1)=(10)^\omega$. As explained above, since $\qDBi{1}(1)$ is purely periodic, we consider  the writing $1(01)^\omega$ instead of $(10)^\omega$. We obtain the Büchi automaton depicted in Figure~\ref{Fig : Automaton-BB-1+Sqrt13-accessible}. 
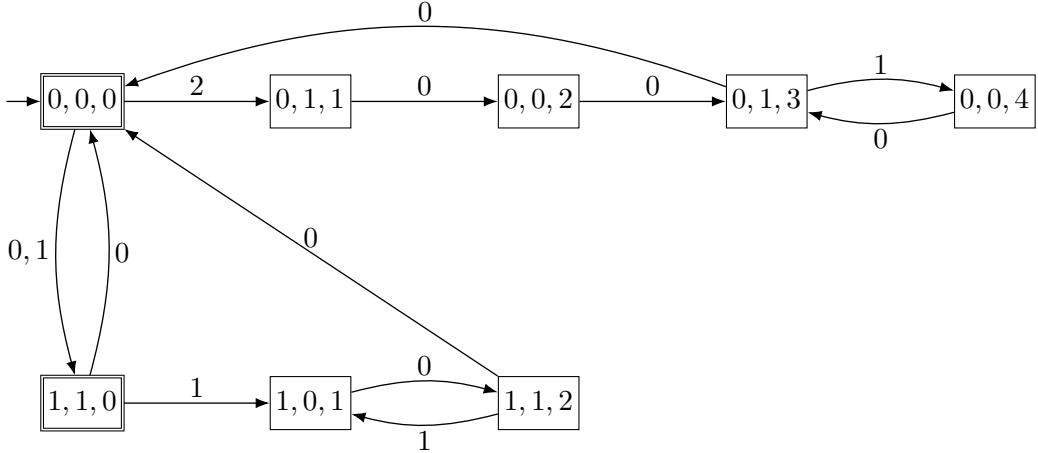
\begin{figure}[htb]
\centering
\begin{tikzpicture}
\tikzstyle{every node}=[shape=rectangle, fill=none, draw=black,
minimum size=20pt, inner sep=2pt]
\node[accepting] (1) at (0,4) {$0,0,0$};
\node (2) at (3,4) {$0,1,1$};
\node  (3) at (6,4) {$0,0,2$};
\node (4) at (9,4) {$0,1,3$};
\node (5) at (12,4) {$0,0,4$};

\node[accepting] (6) at (0,0) {$1,1,0$};
\node (7) at (3,0) {$1,0,1$};
\node (8) at (6,0) {$1,1,2$};

\tikzstyle{every node}=[shape=rectangle, fill=none, draw=black, minimum size=15pt, inner sep=2pt]
\tikzstyle{every path}=[color=black, line width=0.5 pt]
\tikzstyle{every node}=[shape=rectangle, minimum size=5pt, inner sep=2pt]
\draw [-Latex] (-1,4) to node [above] {$ $} (1); 

\draw [-Latex] (1) to node [above] {$2$} (2); 
\draw [-Latex] (1) to [bend right=15] node [left] {$0,1$} (6); 

\draw [-Latex] (2) to node [above] {$0$} (3); 

\draw [-Latex] (3) to node [above] {$0$} (4); 

\draw [-Latex] (4) to [bend left=15] node [above] {$1$} (5); 
\draw [-Latex] (4) to [bend right=20] node [above] {$0$} (1); 

\draw [-Latex] (5) to [bend left=15] node [below] {$0$} (4); 

\draw [-Latex] (6) to node [above] {$1$} (7); 
\draw [-Latex] (6) to [bend right=15] node [right] {$0$} (1); 

\draw [-Latex] (7) to [bend left=15] node [above] {$0$} (8); 
\draw [-Latex] (8) to [bend left=15] node [below] {$1$} (7); 
\draw [-Latex] (8) to node [above] {$0$} (1); 

\end{tikzpicture}
\caption{A Büchi automaton accepting the set $\{\DB(x) : x\in[0,1)\}$ for $\B=(\frac{1+ \sqrt{13}}{2},\frac{5+ \sqrt{13}}{6})$.}
\label{Fig : Automaton-BB-1+Sqrt13-accessible}
\end{figure}
Following the same steps as described in the proof of Theorem~\ref{Thm : Normalization}, from the automata depicted in Figures~\ref{Fig : ZeroAutomaton-1+Sqrt13-accessible} and~\ref{Fig : Automaton-BB-1+Sqrt13-accessible}, we obtain a finite Büchi automaton computing the normalization function in base $\B$ over the pair of alphabets $\boldsymbol{D}=([\![-2,2]\!],[\![-1,1]\!])$.
\end{example}

\section{Further work}

We have shown that properties of alternate base numeration systems defined using a $p$-tuple of bases $(\beta_0,\dots,\beta_{p-1})$ are related to the geometry of a generalized Erd\H os spectrum of the number $\pr=\prod_{i=0}^{p-1}\beta_i$. For $p=1$, i.e., for classical numeration systems with one base $\beta$, the Erd\H os spectrum proved to be useful in different situations. V\'avra in~\cite{Vavra2021} used the spectrum to characterize complex bases $\beta$ for which, with a suitably chosen digit set, every element of the field $\Q(\beta)$ has an eventually periodic representation. In~\cite{HareMasakovaVavra2018} the question whether a numeration system with a complex base $\beta$ and a digit set $\Dig\subset\C$ allows a representation of any complex number is reformulated as the question whether the corresponding Erd\H os spectrum is relatively dense. Let us mention that, so far, the topological properties of the spectrum described in Theorem~\ref{Thm : AccPointAkiyamaKomornikFeng} are known only when the base $\beta$ is a real number greater than $1$ and the alphabet $\Dig$ is a symmetric set of consecutive integers, i.e., of the form $[\![-d,d]\!]$ for some $d\in\N$. In particular, in this case, we have a sharp bound on $d$ for which the spectrum of a non-Pisot base $\beta>1$ has an accumulation point: the spectrum $X^d(\beta)$ has an accumulation point in $\R$ if and only if $d<\beta-1$. Analogous results for real and complex bases $\beta$ and arbitrary finite alphabets $\Dig\subset\R$, or $\Dig\subset\C$, would improve bounds on digit sets in several diverse problems, namely also in our Theorem~\ref{Thm : MainEquivalences}. For more details, see \cite{FrougnyPelantova2018}.

One of the results of this paper is that if $\pr=\prod_{i=0}^{p-1}\beta_i$ is a Pisot number and $\beta_0,\ldots,\beta_{p-1}$ belong to $\Q(\pr)$, the expansions $\DBi{i}(1)$ are all eventually periodic. We have illustrated that $\pr$ being Pisot is not a necessary condition. For $p=1$, bases $\beta$ for which $1$ has an eventually periodic greedy expansion are called Parry numbers. Solomyak obtained algebraic properties of Parry numbers~\cite{Solomyak1994}. It would be interesting to study the analogy of Parry numbers in the context of alternate bases. In particular, to find bounds on the algebraic conjugates of $\pr$. 

In the case where $p=1$, for every sequence of non-negative digits $a=a_0a_1a_2\cdots$ satisfying the lexicographic condition $a_na_{n+1}a_{n+2}\cdots \le_{\lex} a$ for all $n\in\N$, there exists a unique $\beta>1$ such that $d_\beta(1)=a$~\cite{Parry1960}. It is not clear yet whether for $p$ integer digit sequences $a^{(0)}, \dots, a^{(p-1)}$ satisfying analogous lexicographic conditions, there exists a unique $p$-tuple of bases $\beta_0,\dots,\beta_{p-1}$ such that $\DBi{i}(1)=a^{(i)}$ for $i=0,1,\dots,p-1$. Corollary~\ref{Cor : UniquenessDB} represents a first step towards this direction.

\section{\textbf{Acknowledgment}}
Émilie Charlier is supported by the FNRS grant J.0034.22.
Célia Cisternino is supported by the FNRS grant 1.A.564.19F.
Zuzana Mas\'akov\'a and Edita Pelantov\'a are supported by the European Regional Development Fund project CZ.02.1.01/0.0/0.0/16\_019/0000778.

\bibliographystyle{abbrv}
\bibliography{Normalization}

\end{document}